\documentclass[10pt]{amsart}

\usepackage{amsmath}
\usepackage{enumerate}
\usepackage{hyperref}
\usepackage{amsthm}
\usepackage{amssymb}
\usepackage{amsfonts}
\usepackage{tikz} 
\usepackage{yfonts}
\usepackage{mathtools}
\usepackage{mathpazo}
\usepackage{setspace}
\usepackage{cancel}
\usepackage{chemarrow}
\usetikzlibrary{matrix,arrows} 
\usepackage{amscd}
\numberwithin{equation}{section}
\newtheorem{prop}{Proposition}[section]
\newtheorem*{prop*}{Proposition}
\newtheorem{lemma}[prop]{Lemma}
\newtheorem*{lemma*}{Lemma}
\newtheorem*{defn*}{Definition}
\newtheorem{theorem}[prop]{Theorem}
\newtheorem{cor}[prop]{Corollary}
\newtheorem*{cor*}{Corollary}
\newtheorem{defn}[prop]{Definition}

\newtheorem{remark}[prop]{Remark}
\begin{document}
\title[Wallcrossing and higher rank JS pairs]{Wall-crossing and invariants of higher rank Joyce--Song stable pairs}
\author{Artan Sheshmani}
\address{600 Math tower, 231 West 18th Avenue, Ohio State University, Columbus, Ohio, 43210}
\email{sheshmani.1@math.osu.edu}
\date{\today}                                   

\thanks{The author acknowledges partial support from NSF grants DMS 0244412, DMS 0555678 and DMS 08-38434 EMSW21-MCTP (R.E.G.S)}


\begin{abstract}
We introduce a higher rank analog of the Joyce--Song theory of stable pairs. Given a nonsingular projective Calabi--Yau threefold $X$, we define the higher rank Joyce--Song pairs given by ${O}^{\oplus r}_{X}(-n)\rightarrow F$ where $F$ is a pure coherent sheaf with one dimensional support, $r>1$ and $n\gg 0$ is a fixed integer. We equip the higher rank pairs with a Joyce--Song stability condition and compute their associated invariants using the wallcrossing techniques in the category of ``\textit{weakly}" semistable objects. 
\end{abstract}

\maketitle

\section{Introduction}
The Donaldson--Thomas theory of a Calabi--Yau threefold $X$ is defined in \cite{a24} and \cite{a20} via integration against the virtual fundamental class of the moduli space of ideal sheaves. In \cite{a18} and \cite{a17} Pandharipande and Thomas introduced objects given by pairs $(F,s)$ where $F$ is a pure sheaf with one dimensional support together with a fixed Hilbert polynomial and $s\in H^{0}(X,F)$ is given as a section of $F$. The authors computed the invariants of stable pairs using deformation theory and virtual fundamental classes.

Following their work, Joyce and Song defined a similar notion of a (\textit{twisted}) stable pair given by a sheaf $F$ and section map $s:\mathcal{O}(-n)\rightarrow F$ where $n\gg 0$ was chosen to be a sufficiently large integer so that the cohomology vanishing condition $H^{1}(F(n))=0$ is satisfied. These stable pairs were equipped with a stability condition rather different than the one used in \cite{a17}:

\begin{defn}\label{newstabl} 
(Joyce--Song pair stability) Given a coherent sheaf $F$ let $p_{F}$ denote the reduced Hilbert polynomial of $F$ with respect to the ample line bundle $\mathcal{O}_{X}(1)$. A pair $\phi:\mathcal{O}\rightarrow F$ is called $\Hat{\tau}$-stable if the following conditions are satisfied:
\begin{enumerate}
\item $p_{F'}\leq p_{F}$ for all proper subsheaves $F'$ of $F$ such that $F'\neq0$.
\item If $\phi$ factors through $F'$ ($F'$ a proper subsheaf of $F$), then $p_{F'}<p_{F}$.
\end{enumerate}
\end{defn}
In this article we refer to this stability as $\Hat{\tau}$-stability. For more on Joyce--Song stability look at \cite[Definition 12.2]{a30}.

In \cite{a30} Joyce and Song then described the advantage of defining their notion of stable pairs, which was to compute the ``\textit{Generalized Donaldson--Thomas invariants}" by means of (pair) invariants for $\hat{\tau}$-stable pairs. The generalized Donaldson--Thomas invariants could not be calculated using the machinery developed by Thomas in \cite{a20}, as they were given by invariants of semistable sheaves (not just the stable ones). After work of Joyce and Song the interesting question was to whether one is able to study and compute the invariants of objects composed of a sheaf $F$ and multiple sections given by the morphism $s_{1}\cdots s_{r}:\mathcal{O}^{\oplus r}(-n)\rightarrow F$ for $r>1$. The main purpose of the current article is to compute the invariants of these ``\textit{higher rank}" $\hat{\tau}$-stable Joyce--Song pairs with respect to the generalized Donaldson--Thomas invariants, using the method of wallcrossing.

\textbf{The main idea:} The general philosophy in this article is to define an auxiliary category $\mathcal{B}_{p}$ \cite[Section 13.1]{a30}. The objects in $\mathcal{B}_{p}$ are defined similar to the higher rank Joyce--Song pairs and they are classified based on their numerical class $(\beta,r)$. Here, $\beta$ denotes the Chern character of $F$ and $r$ denotes the number of sections of $F$ being considered in the construction, and the definition of the category $\mathcal{B}_{p}$ allows one to define ``\textit{weak}" stability conditions  (say) $\tau^{\bullet}$ and $\tilde{\tau}$ on $\mathcal{B}_{p}$ (look at Definition \ref{weak}).

As we will show  in Theorem \ref{HFT-Bp} and Corollary \ref{N}, the moduli stack of $\tilde{\tau}$-semistable objects in $\mathcal{B}_{p}$ is closely related to the parameterizing moduli stack of higher rank $\hat{\tau}$-semistable Joyce--Song pairs, which enables us to define the invariants of $\hat{\tau}$-semistable pairs of rank $r>1$, $\textbf{N}^{\beta,r}_{stp}(\hat{\tau})$, as$$\textbf{N}^{\beta,r}_{stp}(\hat{\tau})=(-1)^{r^2}\textbf{B}_{p}^{ss}(X,\beta, r, \tilde{\tau}).$$Here $\textbf{B}_{p}^{ss}(X,\beta, r, \tilde{\tau})$ are the invariants of $\tilde{\tau}$-semistable objects in $\mathcal{B}_{p}$ of class $(\beta,r)$ and the sign $(-1)^{r^2}$ is explained in Corollary \ref{N}. 

On the other hand, changing the weak stability condition from $\tau^{\bullet}$ to $\tilde{\tau}$ and using the machinery of the Ringel-Hall algebra of the category $\mathcal{B}_{p}$ discussed in \cite{a30}, provides us with a wallcrossing identity, which relates the invariants of $\tilde{\tau}$-semistable objects in $\mathcal{B}_{p}$ to the invariants of $\tau^{\bullet}$-semistable objects. Moreover, by Proposition \ref{stackprop}, the invariants of $\tau^{\bullet}$-semistable objects are equal to the generalized Donaldson--Thomas invariants of Gieseker semistable sheaves  ($\tau$-semistable for short). 

Therefore, using the above two correposndences, enables us to eventually describe $\textbf{N}^{\beta,r}_{stp}(\hat{\tau})$ with respect to the generalized Donaldson--Thomas invariants of $\tau$-semistable sheaves.  We emphasize here that, though this approach can be used conceptually for all ranks $r>1$, in this article we perform the wallcrossing calculation only for rank $2$ pairs given by $\mathcal{O}^{\oplus 2}(-n)\rightarrow F$. It turns out (c.f. Corollary \ref{he-lol}) that when $r=2$, $\textbf{N}^{\beta,2}_{stp}(\Hat{\tau})=\textbf{B}^{ss}_{p}(X,\beta,2,\tilde{\tau})$ and hence by Corollary \ref{he-lol} and Equation \eqref{wall-cross}, we obtain the following identity between the invariants, $\textbf{N}^{\beta,2}_{stp}(\Hat{\tau})$, and the generalized Donaldson--Thomas invariants, $\overline{DT}^{\beta_{i}}(\tau)$:
\begin{align*}\label{yet"}
&
\textbf{N}^{\beta,2}_{stp}(\Hat{\tau})=
\sum_{1\leq l,\beta_{1}+\cdots+\beta_{l}=\beta}\frac{-1}{4}\cdot\frac{(1)}{l!}\cdot\prod_{i=1}^{l}\bigg(\overline{DT}^{\beta_{i}}(\tau)\cdot \bar{\chi}_{\mathcal{B}_{p}}((\beta_{1}+\cdots+\beta_{i-1},2),(\beta_{i},0))\notag\\
 &
\cdot (-1)^{\sum_{i=1}^{l}\bar{\chi}_{\mathcal{B}_{p}}((\beta_{1}+\cdots \beta_{i-1},2),(\beta_{i},0))}\bigg).\notag\\
\end{align*}
The above identity can be regarded as the definition of the rank 2 Joyce--Song stable pair invariants of class $(\beta,2)$ (Corollary \ref{he-lol} and Equation \eqref{define-it}).
\subsection*{Acknowledgments}
The author thanks  Sheldon Katz, Yukinobu Toda, Richard Thomas, Emanuel Diaconescu for their invaluable help. Moreover, the author sincerely thanks the referee's valuable comments which helped to highly improve the content and the presentation of the paper.
\section{The auxiliary category $\mathcal{B}_{p}$}\label{wallcross}
\begin{defn}\label{Ap} 
Let $X$ be a nonsingular projective Calabi--Yau threefold equipped with ample line bundle $\mathcal{O}_{X}(1)$. Let $\tau$ denote the Gieseker stability condition on the abelian category of coherent sheaves on $X$. Define $\mathcal{A}_{p}$ to be the sub-category of coherent sheaves whose objects are zero sheaves and non-zero $\tau$-semistable sheaves with fixed reduced Hilbert polynomial $p$ \footnote {Look at \cite[Definition 13.1]{a30} for more detail}. 
\end{defn}
\begin{defn}\label{Bp}
Fix an integer $n$. Now define the category $\mathcal{B}_{p}$ to be the category whose objects are triples $(F,V,\phi)$, where $F\in Obj(\mathcal{A}_{p})$, $V$ is a finite \text{\text{dim}}ensional $\mathbb{C}$-vector space, and $\phi: V\rightarrow \operatorname{Hom}(\mathcal{O}_{X}(-n),F)$ is a $\mathbb{C}$-linear map. Given $(F,V,\phi)$ and $(F',V',\phi')$ in $\mathcal{B}_{p}$ define morphisms $(F,V,\phi)\rightarrow (F',V',\phi')$ in $\mathcal{B}_{p}$ to be pairs of morphisms $(f,g)$ where $f:F\rightarrow F'$ is a morphism in $\mathcal{A}_{p}$ and $g:V\rightarrow V'$ is a $\mathbb{C}$-linear map, such that the following diagram commutes:
\begin{equation*}
\begin{tikzpicture}
back line/.style={densely dotted}, 
cross line/.style={preaction={draw=white, -, 
line width=6pt}}] 
\matrix (m) [matrix of math nodes, 
row sep=2em, column sep=3.25em, 
text height=1.5ex, 
text depth=0.25ex]{ 
V&\operatorname{Hom}(\mathcal{O}_{X}(-n),F)\\
V'&\operatorname{Hom}(\mathcal{O}_{X}(-n),F')\\};
\path[->]
(m-1-1) edge node [above] {$\phi$} (m-1-2)
(m-1-1) edge node [left] {$g$} (m-2-1)
(m-1-2) edge node [right] {$f$}(m-2-2)
(m-2-1) edge node [above] {$\phi'$} (m-2-2);
\end{tikzpicture}
\end{equation*} 
\end{defn}
Our definition of the category $\mathcal{B}_{p}$ is compatible with that of \cite[Definition 13.1]{a30}. \\

Now we define the numerical class of objects in $\mathcal{B}_{p}$ based on \cite[Section 3.1]{a30}. 
\begin{defn}
Define the Grothendieck group $K(\mathcal{B}_{p})=K(\mathcal{A}_{p})\oplus \mathbb{Z}$ where $K(\mathcal{A}_{p})$ is given by the image of $K_{0}(\mathcal{A}_{p})$ in $K(\operatorname{Coh}(X)):=K^{num}(\operatorname{Coh}(X))$. Let  $\mathcal{C}(\mathcal{A}_{p})$ denote the positive cone of $\mathcal{A}_{p}$ defined as$$\mathcal{C}(\mathcal{A}_{p})=\{[E]\in K^{num}(\mathcal{A}_{p}):0\neq E\in \mathcal{A}_{p}\}.$$Now given $(F,V,\phi)\in \mathcal{B}_{p}$, we write $[(F,V,\phi)]=([F],\text{\text{dim}}(V))$ and define the positive cone of $\mathcal{B}_{p}$ by:
\begin{center}
$\mathcal{C}(\mathcal{B}_{p})=\{(\beta,d)| \beta\in \mathcal{C}(\mathcal{A}_{p})$ and $d\geq 0$ or $\beta=0$ and $d>0\},$
\end{center}
\end{defn}
We state the following results by Joyce and Song without proof:
\begin{lemma}
\cite[Lemma 13.2]{a30}. The category $\mathcal{B}_{p}$ is abelian and $\mathcal{B}_{p}$ satisfies the condition that If $[F]=0\in K(\mathcal{A}_{p})$ then $F\cong 0$. Moreover, $\mathcal{B}_{p}$ is noetherian and artinian and the moduli stacks $\mathfrak{M}^{(\beta,d)}_{\mathcal{B}_{p}}$ are of finite type for all $(\beta,d)\in C(\mathcal{B}_{p})$. 
\end{lemma}
\begin{remark}
The category $\mathcal{A}_{p}$ embeds as a full and faithful sub-category in $\mathcal{B}_{p}$ by $F\rightarrow (F,0,0)$. Moreover, it is shown in \cite[Equation (13.3)]{a30} that every object $(F,V,\phi)$ sits in a short exact sequence.
\begin{equation}\label{coreexact}
0\rightarrow (F,0,0)\rightarrow (F,V,\phi)\rightarrow (0,V,0)\rightarrow 0
\end{equation}
\end{remark}
Next we recall the definition of ``\textit{weak}" (semi)stability for a general abelian category $\mathcal{A}$. 
\begin{defn}\label{weakjoyce}
\footnote{For more detail on Definition \ref{weakjoyce} look at \cite[Definition 3.5]{a30}.} 
Let $\mathcal{A}$ be an abelian category. Let $K(\mathcal{A})$ be the quotient of $K_{0}(\mathcal{A})$ by some fixed group. Let $$C(\mathcal{A}):=\{[E]\in K^{num}(\mathcal{A}): 0 \neq E \in \mathcal{A}\}$$ be the positive cone of $\mathcal{A}$. Suppose $(T,\leq)$ is a totally ordered set and $\tau: C(\mathcal{A})\rightarrow T$ a map. We call $(\tau,T,\leq)$ a stability condition on $\mathcal{A}$ if whenever $\alpha,\beta,\gamma\in C(\mathcal{A})$ with $\beta=\alpha+\gamma$ then either $$\tau(\alpha)<\tau(\beta)<\tau(\gamma)$$ or $$\tau(\alpha)>\tau(\beta)>\tau(\gamma)$$ or $\tau(\alpha)=\tau(\beta)=\tau(\gamma)$. We call $(\tau,T,\leq)$ a weak stability condition on $\mathcal{A}$ if whenever $\alpha,\beta,\gamma\in C(\mathcal{A})$ with $\beta=\alpha+\gamma$ then either $\tau(\alpha)\leq\tau(\beta)\leq\tau(\gamma)$ or $\tau(\alpha)\geq\tau(\beta)\geq\tau(\gamma)$. For such $(\tau,T,\leq)$, we say that a nonzero object $E$ in $\mathcal{A}$ is
\begin{enumerate}
\item $\tau$-semistable if $\forall S\subset E$ where $S\ncong 0$, we have $\tau([S])\leq \tau([E/S])$
\item $\tau$-stable if $\forall S\subset E$ where $S\ncong 0$, we have $\tau([S])< \tau([E/S])$
\item $\tau$-unstable if it is not $\tau$-semistable.
\end{enumerate}
\end{defn}
Now we apply the definition of weak stability conditions to the category $\mathcal{B}_{p}$: 
\begin{defn}\label{weak}
Define the weak stability conditions $(\tau^{\bullet}, T^{\bullet}, \leq)$, $(\tilde{\tau},\tilde{T}, \leq)$ and $(\tau^{n}, T^{n}, \leq)$ on $\mathcal{B}_{p}$  by:
\begin{enumerate}
\item $T^{\bullet}=\{-1,0\}$ with the natural order $-1<0$, and $\tau^{\bullet}(\beta,d)=0$ if $d=0$ and $\tau^{\bullet}(\beta,d)=-1$ if $d>0$.
\item $\tilde{T}=\{0, 1\}$ with the natural order $0<1$, and $\tilde{\tau}(\beta,d)=0$ if $d=0$ and $\tilde{\tau}(\beta,d)=1$ if $d>0$.
\item $T^{n}=\{0\}$, and $\tau^{n}(\beta,d)=0$ $\forall(\beta,d)$.
\end{enumerate}
\end{defn}
Definition \ref{weak} is compatible with that of \cite[Definition. 13.5]{a30}.
\section*{Moduli stack of objects in $\mathcal{B}_{p}$}
Now we describe the moduli stack of weakly semistable objects in $\mathcal{B}_{p}$. We construct this moduli stack only for the $\tilde{\tau}$-semistability condition as the constructions for $\tau^{\bullet}$-semistability is similar and left to the reader. 

\begin{remark}\label{better}
In order to better facilitate the understanding of our strategy in this article, we emphasize that $\tilde{\tau}$ and $\tau^{\bullet}$ stability conditions are only defined as weak stability conditions on the auxiliary category $\mathcal{B}_{p}$ and so it remains to prove that  $\tau^{\bullet}$-semistable and $\tilde{\tau}$-semistable objects in $\mathcal{B}_{p}$ are related to  $\tau$-semistable sheaves and $\hat{\tau}$-semistable pairs respectively. 
For the former, we use the fact that by \cite[Proposition 13.6]{a30} the $\tau^{\bullet}$-semistable objects in $\mathcal{B}_{p}$ are equivalent to $\tau$-semistable sheaves and then we prove the correspondence between $\hat{\tau}$-semistable pairs and $\tilde{\tau}$-semistable objects in Theorem \ref{HFT-Bp}.
\end{remark}

\subsection*{Objects in $\mathcal{B}_{p}$ as complexes in the derived category}

\begin{remark}\label{N-E-F}
By \cite[Lemma 13.2]{a30} there exists a natural embedding functor $\mathfrak{F}:\mathcal{B}_{p}\rightarrow D(X)$ which takes $(F,V,\phi_{V})\in \mathcal{B}_{p}$ to an object in the derived category given by $\cdots\rightarrow 0\rightarrow V\otimes \mathcal{O}_{X}(-n)\rightarrow F\rightarrow 0\rightarrow \cdots$ where $V\otimes \mathcal{O}_{X}(-n)$ and $F$ sit in degree $-1$ and $0$. Assume that $\operatorname{dim}(V)=r$. In that case $V\otimes \mathcal{O}_{X}(-n)\cong \mathcal{O}_{X}(-n)^{\oplus r}$. Therefore, one may view an object  $(F,V,\phi_{V})\in \mathcal{B}_{p}$ as an object in an abelian subcategory of the derived category, $[\mathcal{O}_{X}(-n)^{\oplus r}\rightarrow F]$.
\end{remark}
\begin{defn}\label{Bp-sflat}
Fix a parameterizing scheme of finite type $S$. Let $\pi_{X}:X\times S\rightarrow X$ and $\pi_{S}:X\times S\rightarrow S$ denote the natural projections. Use the natural embedding functor $\mathfrak{F}:\mathcal{B}_{p}\rightarrow D(X)$ in Remark \ref{N-E-F}. Define the $S$-flat family of objects in $\mathcal{B}_{p}$ of type $(\beta,r)$ as a complex $$\pi_{S}^{*}M\otimes \pi_{X}^{*}\mathcal{O}_{X}(-n)\xrightarrow{\psi_{S}} \mathcal{F}$$ sitting in degree $-1$ and $0$ such that $\mathcal{F}$ is given by an $S$-flat family of semistable sheaves with fixed reduced Hilbert polynomial $p$ with $\operatorname{ch}(F)=\beta$ and $M$ is a vector bundle of rank $r$ over $S$. A morphism between two such $S$-flat families is given by a morphism between the complexes $\pi_{S}^{*}M\otimes \pi_{X}^{*}\mathcal{O}_{X}(-n)\xrightarrow{\psi_{S}} \mathcal{F}$ and $\pi_{S}^{*}M'\otimes \pi_{X}^{*}\mathcal{O}_{X}(-n)\xrightarrow{\psi'_{S}} \mathcal{F}'$:
\begin{center}
\begin{tikzpicture}
back line/.style={densely dotted}, 
cross line/.style={preaction={draw=white, -, 
line width=6pt}}] 
\matrix (m) [matrix of math nodes, 
row sep=2em, column sep=3.5em, 
text height=1.5ex, 
text depth=0.25ex]{ 
\pi_{S}^{*}M\otimes \pi_{X}^{*}\mathcal{O}_{X}(-n)&\mathcal{F}\\
\pi_{S}^{*}M'\otimes \pi_{X}^{*}\mathcal{O}_{X}(-n)&\mathcal{F}'.\\};
\path[->]
(m-1-1) edge node [above] {$\psi_{S}$} (m-1-2)
(m-1-1) edge (m-2-1)
(m-1-2) edge (m-2-2)
(m-2-1) edge node [above] {$\psi'_{S}$} (m-2-2);
\end{tikzpicture}
\end{center}
Moreover an isomorphism between two such $S$-flat families in $\mathcal{B}_{p}$ is given by an isomorphism between the associated complexes $\pi_{S}^{*}M\otimes \pi_{X}^{*}\mathcal{O}_{X}(-n)\xrightarrow{\psi_{S}} \mathcal{F}$ and $\pi_{S}^{*}M'\otimes \pi_{X}^{*}\mathcal{O}_{X}(-n)\xrightarrow{\psi'_{S}} \mathcal{F}'$:
\begin{center}
\begin{tikzpicture}
back line/.style={densely dotted}, 
cross line/.style={preaction={draw=white, -, 
line width=6pt}}] 
\matrix (m) [matrix of math nodes, 
row sep=2em, column sep=3.5em, 
text height=1.5ex, 
text depth=0.25ex]{ 
\pi_{S}^{*}M\otimes \pi_{X}^{*}\mathcal{O}_{X}(-n)&\mathcal{F}\\
\pi_{S}^{*}M'\otimes \pi_{X}^{*}\mathcal{O}_{X}(-n)&\mathcal{F}'.\\};
\path[->]
(m-1-1) edge node [above] {$\psi_{S}$} (m-1-2)
(m-1-1) edge node [left] {$\cong$} (m-2-1)
(m-1-2) edge node [right] {$\cong$} (m-2-2)
(m-2-1) edge node [above] {$\psi'_{S}$} (m-2-2);
\end{tikzpicture}
\end{center}
From now on, by objects in $\mathcal{B}_{p}$ we mean the objects which lie in the image of the natural embedding functor $\mathfrak{F}:\mathcal{B}_{p}\rightarrow D(X)$ in Remark \ref{N-E-F}. Moreover, by the $S$-flat family of objects in $\mathcal{B}_{p}$, their morphisms (or isomorphisms) we mean their corresponding definitions as stated in Definition \ref{Bp-sflat}. 
\end{defn}
Now we define the \textit{rigidified} objects in $\mathcal{B}_{p}$. We will not perform wallcrossing computation for these objects and they are only going to provide the means for construction of moduli stack of objects in $\mathcal{B}_{p}$ as a quotient stack.
\subsection{Rigidified objects and their realization in the derived category}
As stated in Definition \ref{Bp-sflat}, the objects in $\mathcal{B}_{p}$ are defined such that the sheaf sitting in degree $-1$ is given by a trivial vector bundle of rank $r$ isomorphic to $\mathcal{O}^{\oplus r}_{X}(-n)$. However we have not fixed any choice of such trivialization. Below we will define the closely related objects, which we denote by rigidified objects in $\mathcal{B}_{p}$, by fixing a choice of the trivialization of $\mathcal{O}^{\oplus r}_{X}(-n)$. These objects are essential for our construction, as their moduli stack forms a $GL_{r}(\mathbb{C})$-torsor over the (to be defined)  moduli stack of objects in $\mathcal{B}_{p}$. Therefore, the moduli stack of objects in $\mathcal{B}_{p}$ is obtained by taking the stacky quotient of the moduli stack of rigidified objects by the action of $GL_{r}(\mathbb{C})$.

\begin{defn}\label{rigid-Bp}
Fix a positive integer $r$ and define the subcategory $\mathcal{B}_{p}^{\textbf{R}}\subset \mathcal{B}_{p}$ to be the category of ``rigidified" objects in $\mathcal{B}_{p}$ of rank $r$ whose objects are defined by tuples $(F,\mathbb{C}^{\oplus r},\rho)$ where $F$ is a $\tau$-semistable coherent sheaf with reduced Hilbert polynomial $p$ and $\operatorname{ch}(F)=\beta$ and $\rho:\mathbb{C}^{r}\rightarrow \operatorname{Hom}(\mathcal{O}_{X}(-n),F)$. Given two rigidified objects of fixed type $(\beta,r)$ as $(F,\mathbb{C}^{\oplus r},\rho)$ and $(F',\mathbb{C}^{\oplus r},\rho')$ in $\mathcal{B}_{p}^{\textbf{R}}$, define morphisms $(F,\mathbb{C}^{\oplus r},\rho)\rightarrow (F',\mathbb{C}^{\oplus r},\rho')$ to be given by a morphism $f:F\rightarrow F'$ in $\mathcal{A}_{p}$ such that the following diagram commutes:
\begin{equation*}
\begin{tikzpicture}
back line/.style={densely dotted}, 
cross line/.style={preaction={draw=white, -, 
line width=6pt}}] 
\matrix (m) [matrix of math nodes, 
row sep=2em, column sep=3.25em, 
text height=1.5ex, 
text depth=0.25ex]{ 
\mathbb{C}^{\oplus r}&\operatorname{Hom}(\mathcal{O}_{X}(-n),F)\\
\mathbb{C}^{\oplus r}&\operatorname{Hom}(\mathcal{O}_{X}(-n),F').\\};
\path[->]
(m-1-1) edge node [above] {$\rho$} (m-1-2) 
(m-2-1) edge node [above] {$\rho'$} (m-2-2)
(m-1-2) edge node [right] {$f$}(m-2-2)
(m-1-1) edge node [left] {$\operatorname{id}$} (m-2-1);
\end{tikzpicture}
\end{equation*} 
\end{defn}
\begin{remark}\label{embed-func-BpR}
There exists a natural embedding functor $\mathfrak{F}^{\textbf{R}}:\mathcal{B}^{\textbf{R}}_{p}\rightarrow D(X)$ which takes $(F,\mathbb{C}^{\oplus r},\rho)\in \mathcal{B}^{\textbf{R}}_{p}$ to an object in the derived category given by $\cdots\rightarrow 0\rightarrow \mathbb{C}^{\oplus r}\otimes \mathcal{O}_{X}(-n)\rightarrow F\rightarrow 0\rightarrow \cdots$ where $\mathbb{C}^{\oplus r}\otimes \mathcal{O}_{X}(-n)$ sits in degree $-1$ and $F$ sits in degree $0$. One may view an object in $\mathcal{B}^{\textbf{R}}_{p}$ as a complex $\phi:\mathcal{O}^{\oplus r}_{X}(-n)\rightarrow F$ such that the choice of trivialization of $\mathcal{O}^{\oplus r}_{X}(-n)$ is fixed.
\end{remark}
\begin{defn}\label{Bp-R-sflat}
Fix a parametrizing scheme of finite type $S$. Use the natural embedding functor $\mathfrak{F}^{\textbf{R}}:\mathcal{B}^{\textbf{R}}_{p}\rightarrow D(X)$ in Remark \ref{embed-func-BpR}. An $S$-flat family of objects of type $(\beta,r)$ in $\mathcal{B}^{\textbf{R}}_{p}$ is given by a complex $$\pi_{S}^{*}\mathcal{O}_{S}^{\oplus r}\otimes \pi_{X}^{*}\mathcal{O}_{X}(-n)\xrightarrow{\psi_{S}} \mathcal{F}$$ sitting in degree $-1$ and $0$ such that $\mathcal{F}$ is given by an $S$-flat family of semistable sheaves with fixed reduced Hilbert polynomial $p$ with $\operatorname{ch}(\mathcal{F}_{s})=\beta$ for all $s\in S$. A morphism between two such $S$-flat families in $\mathcal{B}^{\textbf{R}}_{p}$ is given by a morphism between the complexes $\pi_{S}^{*}\mathcal{O}_{S}^{\oplus r}\otimes \pi_{X}^{*}\mathcal{O}_{X}(-n)\xrightarrow{\psi_{S}} \mathcal{F}$ and $\pi_{S}^{*}\mathcal{O}_{S}^{\oplus r}\otimes \pi_{X}^{*}\mathcal{O}_{X}(-n)\xrightarrow{\psi'_{S}} \mathcal{F}'$:
\begin{center}
\begin{tikzpicture}
back line/.style={densely dotted}, 
cross line/.style={preaction={draw=white, -, 
line width=6pt}}] 
\matrix (m) [matrix of math nodes, 
row sep=2em, column sep=3.5em, 
text height=1.5ex, 
text depth=0.25ex]{ 
\pi_{S}^{*}\mathcal{O}_{S}^{\oplus r}\otimes \pi_{X}^{*}\mathcal{O}_{X}(-n)&\mathcal{F}\\
\pi_{S}^{*}\mathcal{O}_{S}^{\oplus r}\otimes \pi_{X}^{*}\mathcal{O}_{X}(-n)&\mathcal{F}'.\\};
\path[->]
(m-1-1) edge node [above] {$\psi_{S}$} (m-1-2)
(m-1-1) edge node [left] {$\operatorname{id}_{\mathcal{O}_{X\times S}}$}(m-2-1)
(m-1-2) edge (m-2-2)
(m-2-1) edge node [above] {$\psi'_{S}$} (m-2-2);
\end{tikzpicture}
\end{center}
Moreover an isomorphism between two such $S$-flat families in $\mathcal{B}^{\textbf{R}}_{p}$ is given by an isomorphism between the associated complexes $\pi_{S}^{*}\mathcal{O}_{S}^{\oplus r}\otimes \pi_{X}^{*}\mathcal{O}_{X}(-n)\xrightarrow{\psi_{S}} \mathcal{F}$ and $\pi_{S}^{*}\mathcal{O}_{S}^{\oplus r}\otimes \pi_{X}^{*}\mathcal{O}_{X}(-n)\xrightarrow{\psi'_{S}} \mathcal{F}'$:
\begin{center}
\begin{tikzpicture}
back line/.style={densely dotted}, 
cross line/.style={preaction={draw=white, -, 
line width=6pt}}] 
\matrix (m) [matrix of math nodes,  
row sep=2em, column sep=3.5em, 
text height=1.5ex, 
text depth=0.25ex]{ 
\pi_{S}^{*}\mathcal{O}_{S}^{\oplus r}\otimes \pi_{X}^{*}\mathcal{O}_{X}(-n)&\mathcal{F}\\
\pi_{S}^{*}\mathcal{O}_{S}^{\oplus r}\otimes \pi_{X}^{*}\mathcal{O}_{X}(-n)&\mathcal{F}'.\\};
\path[->]
(m-1-1) edge node [above] {$\psi_{S}$} (m-1-2)
(m-1-1) edge node [left] {$\operatorname{id}_{\mathcal{O}_{X\times S}}$} (m-2-1)
(m-1-2) edge node [right] {$\cong$} (m-2-2)
(m-2-1) edge node [above] {$\psi'_{S}$} (m-2-2);
\end{tikzpicture}
\end{center}
Similar to the way that we treated objects in $\mathcal{B}_{p}$, from now on by objects in $\mathcal{B}^{\textbf{R}}_{p}$ we mean the objects which lie in the image of the natural embedding functor $\mathfrak{F}^{\textbf{R}}:\mathcal{B}^{\textbf{R}}_{p}\rightarrow D(X)$ in Remark \ref{embed-func-BpR}. Moreover by the $S$-flat family of objects in $\mathcal{B}^{\textbf{R}}_{p}$, their morphisms (or isomorphisms) we mean the corresponding definitions as stated in Definition \ref{Bp-R-sflat}. 
\end{defn}
\textbf{Notation}: 
In what follows we define $\mathfrak{M}^{(\beta,r)}_{\mathcal{B}_{p}}$ ($\mathfrak{M}^{(\beta,r)}_{\mathcal{B^{\textbf{R}}}_{p}}$ respectively) to be the moduli functors from $Sch/\mathbb{C}\rightarrow \operatorname{Groupoids}$ which send a $\mathbb{C}$-scheme $S$ to the groupoid of $S$-flat families of objects of type $(\beta,r)$ in $\mathcal{B}_{p}$ ($\mathcal{B}^{\textbf{R}}_{p}$ respectively). \\

We will show that these moduli functors (as groupoid valued functors) are equivalent to algebraic quotient stacks. We will also show that the moduli stack $\mathfrak{M}^{(\beta,r)}_{\mathcal{B}_{p}}$ is given by a stacky quotient of $\mathfrak{M}^{(\beta,r)}_{\mathcal{B}^{\textbf{R}}_{p}}$. 

\subsection{The underlying parameter scheme}\label{scheme-beta}

According to Definition \ref{Bp} an object in the category $\mathcal{A}_{p}$ consists of semistable sheaves with fixed reduced Hilbert polynomial $p$. Note that having fixed a polynomial (in variable $t$) $p(t)$ as the reduced Hilbert polynomial of $F$ means that the Hilbert polynomial of $F$ can yet be chosen as $P_{F}(t)= \frac{k}{d'!} \cdot p(t)$ for different values of $k$  where $d'$ is the dimension of $F$. However here we make an assumption that there are only finitely many possible values $k = 1,\cdots, N$ for which our computation makes sense. We explain the motivation behind this assumption further below.

Our analysis inherits this finiteness property directly from applying \cite[Proposition 13.7]{a30} where the authors show that there are only finitely many nontrivial contributions to the wallcrossing computation which are induced by objects, whose underlying sheaves could only have finitely many fixed Hilbert polynomials. In other words, according to  \cite[Proposition 13.7]{a30}, it suffices to consider Hilbert polynomials $P_{F}(t)= \frac{k}{d'!} \cdot p(t)$ induced by $p$ and only finitely many values of $k=1,2,\cdots, N$ for some $N>0$. 

On the other hand, as discussed in  \cite[Theorem 3.37]{a8}, the family of $\tau$-semistable sheaves $F$ on $X$ such that $F$ has a fixed Hilbert polynomial is bounded. Therefore, the family of coherent sheaves with finitely many fixed Hilbert polynomials is also bounded.  We will use this boundedness property in our construction of parameterizing moduli stacks. 

Now fix the Hilbert polynomial $P_{F}(t)=P$ as above, and use the fact that given a bounded family $\mathcal{F}$ of coherent sheaves with fixed Hilbert polynomial $P$ (here $F$ denotes each member of the family $\mathcal{F}$), there exists an upper bound for their Castelnuvo-Mumford regularity, given by the integer $m$ such that for each member of the family, $F$, the twisted sheaf $F(m')$ is globally generated for all $m'\geq m$. Fix such $m'$ and let $V$ be the complex vector space of \text{\text{dim}}ension $d=P(m')$ given by $V=\operatorname{H}^{0}(F\otimes \mathcal{O}_{X}(m'))$. Twisting the sheaf $F$ by the fixed large enough integer $m'$ would ensure one to get a surjective morphism of coherent sheaves $V\otimes \mathcal{O}_{X}(-m')\rightarrow F$. This defines a closed point $$[V\otimes \mathcal{O}_{X}(-m')\to F]\in \operatorname{Quot}_{P}(V\otimes \mathcal{O}_{X}(-m')),$$where $\operatorname{Quot}_{P}(V\otimes \mathcal{O}_{X}(-m'))$,  known as the Grothendieck's Quot-scheme, is the scheme parametrizing the flat quotients $V\otimes \mathcal{O}_{X}(-m')\to F$ where $F$ has the fixed given Hilbert polynomial $P$. In fact the point $[V\otimes \mathcal{O}_{X}(-m')\to F]$ is contained in an open subset $\mathcal{Q}^{ss}\subset \mathcal{Q}$ of all those quotients $[V\otimes \mathcal{O}_{X}(-m')\to F]$ where $F$ is $\tau$-semistable and the induced map $$H^{0}(V\otimes \mathcal{O}_{X})\rightarrow H^{0}(F(m'))$$ is an isomorphism. For more on construction of Quot schemes  look at \cite[Section 4.3]{a8}.
\begin{defn}\label{scheme-beta-def}
Let $n$ in Definition \ref{Bp-sflat} to be given so that $n\gg m'$. Define $\mathcal{P}$ over $\mathcal{Q}^{ss}$ to be the bundle whose fibers parameterize $\operatorname{H}^{0}(F(n))$. The fibers of the bundle $\mathcal{P}^{\oplus r}$ over each point $[F]=\{p\}$, where $p\in \mathcal{Q}^{ss}$, parameterize $\operatorname{H}^{0}(F(n))^{\oplus r}$.
In other words the fibers of $\mathcal{P}^{\oplus r}$ parameterize the maps $\mathcal{O}_{X}^{\oplus r}(-n)\rightarrow F$ (which define the complexes representing the objects in $\mathcal{B}^{\textbf{R}}_{p}$). 
\end{defn} 
\begin{remark}
Note that here we have used the fact that, by the fixed choices of $n$ and $m'$ satisfying $n\gg m'\geq m$ ($m$ is defined in Section \ref{scheme-beta}) we can see that the dimension of $H^{0}(F(n))$ will remain constant over all points $[V\otimes \mathcal{O}_{X}(-m')\rightarrow F]\in\mathcal{Q}^{ss}$ and therefore, $\mathcal{P}$ remains as a bundle over $\mathcal{Q}^{ss}$.
\end{remark}
There exists a right action of $\operatorname{GL}(V)$ (where $V$ is as above) on  the Quot scheme $\mathcal{Q}$ which induces an action on $\mathcal{Q}^{ss}$, after restriction to the open subscheme of $\tau$-semistable sheaves. It is trivially seen that the action of $\operatorname{GL}(V)$ on $\mathcal{Q}^{ss}$ induces a right action on $\mathcal{P}^{\oplus r}$. Moreover note that, since we have fixed the trivialization of $\mathcal{O}^{\oplus r}_{X}(-n)$ for the objects in $\mathcal{B}^{\textbf{R}}_{p}$, there exists an action of $\operatorname{GL}_{r}(\mathbb{C})$ on $\mathcal{P}^{\oplus r}$ as follows; let $[\mathcal{O}_{X}^{\oplus r}(-n)\xrightarrow{\phi} F]$ be given as a point in $\mathcal{P}^{\oplus r}$. Let $\psi\in \operatorname{GL}_{r}(\mathbb{C})$ be the map given by $\psi: \mathcal{O}_{X}(-n)^{\oplus r}\rightarrow \mathcal{O}_{X}(-n)^{\oplus r}$. The action of $\operatorname{GL}_{r}(\mathbb{C})$ on $\mathcal{P}^{\oplus r}$ is defined via precomposing the sections of $F$ with $\psi$ as shown in the diagram below:
\begin{equation}\label{action}
\begin{tikzpicture}
back line/.style={densely dotted}, 
cross line/.style={preaction={draw=white, -, 
line width=6pt}}] 
\matrix (m) [matrix of math nodes, 
row sep=2em, column sep=3.25em, 
text height=1.5ex, 
text depth=0.25ex]{ 
\mathcal{O}^{\oplus r}_{X}(-n)&\\
\mathcal{O}^{\oplus r}_{X}(-n)&F.\\};
\path[->]
(m-1-1) edge node [left] {$\psi$} (m-2-1)

(m-2-1)  edge node [above]{$\phi$}(m-2-2);
\end{tikzpicture}
\end{equation}
Now use the fact that, by the Grothendieck-Riemann-Roch theorem, fixing a polarization over $X$ and the chern character of sheaves as $\text{ch}(F)=\beta$, induces fixed Hilbert polynomial for such $F$. Therefore from now on, we index our parameterizing moduli stacks by $(\beta,r)$ instead of $(P,r)$. 
\subsection{The Artin stacks $\mathfrak{M}^{(\beta,r)}_{\mathcal{B}_{p}}$ and $\mathfrak{M}^{(\beta,r)}_{\mathcal{B}^{\textbf{R}}_{p}}$.}\label{sec41}
By definitions \ref{Bp-sflat} and \ref{Bp-R-sflat}  the construction of the moduli stack of objects  in $\mathcal{B}_{p}$ and $\mathcal{B}^{\textbf{R}}_{p}$ is done similar to \cite[Section 5]{a39}:

\begin{theorem}\label{theorem81}
Let $\mathcal{P}^{\oplus r}$ be as in Definition \ref{scheme-beta-def}. Then the following statements hold true:
\begin{enumerate}
\item Let $\left[\frac{\mathcal{P}^{\oplus r}}{GL(V)}\right]$ be the stack theoretic quotient of $\mathcal{P}^{\oplus r}$ by $GL(V)$. Then, there exists an isomorphism of stacks $$\mathfrak{M}^{(\beta,r)}_{\mathcal{B^{\textbf{R}}}_{p}}\cong \left[\frac{\mathcal{P}^{\oplus r}}{GL(V)}\right].$$In particular, $\mathfrak{M}^{(\beta,r)}_{\mathcal{B^{\textbf{R}}}_{p}}$ is an Artin stack. 

\item The moduli stack, $\mathfrak{M}^{(\beta,r)}_{\mathcal{B}^{\textbf{R}}_{p}}$, is a $\operatorname{GL}_{r}(\mathbb{C})$-torsor over $\mathfrak{M}^{(\beta,r)}_{\mathcal{B}_{p}}$. 

\item It is true that locally in the flat topology, $\mathfrak{M}^{(\beta,r)}_{\mathcal{B}_{p}}\cong \mathfrak{M}^{(\beta,r)}_{\mathcal{B}^{\textbf{R}}_{p}}\times \left[\frac{\operatorname{Spec}(\mathbb{C})}{\operatorname{GL}_{r}(\mathbb{C})}\right]$. This isomorphism does not hold true globally unless $r=1$.
\end{enumerate}
\end{theorem}
\begin{proof} The proofs of parts (1), (2) are essentially the same as \cite[Proposition 5.5, Corollary 6.4, Theorems 6.2 and Theorem 6.5]{a39}. Now we prove part (3) by showing that there exists a forgetful map $\pi:\mathfrak{M}^{(\beta,r)}_{\mathcal{B}^{\textbf{R}}_{p}}\rightarrow \mathfrak{M}^{(\beta,r)}_{\mathcal{B}_{p}}$ which induces a map from $\mathfrak{M}^{(\beta,r)}_{\mathcal{B}^{\textbf{R}}_{p}}\times \left[\frac{\operatorname{Spec}(\mathbb{C})}{\operatorname{GL}_{r}(\mathbb{C})}\right] $ to $\mathfrak{M}^{(\beta,r)}_{\mathcal{B}_{p}}$ and show that this map has an inverse locally but not globally unless $r=1$. First we prove the claim for $r=1$; 

For $r=1$, $\operatorname{GL}_{1}(\mathbb{C})=\mathbb{G}_{m}$. For a $\mathbb{C}$-scheme $S$, an $S$-point of $\mathfrak{M}^{(\beta,1)}_{\mathcal{B}^{\textbf{R}}_{p}}\times [\frac{\operatorname{Spec}(\mathbb{C})}{\mathbb{G}_{m}}]$ is identified with the data $(\mathcal{O}_{X\times S}(-n)\rightarrow \mathcal{F}, \mathcal{L}_{S})$ where $\mathcal{L}_{S}$ is a $\mathbb{G}_{m}$ line bundle over $S$. Let $\pi_{S}:X\times S\rightarrow S$ be the natural projection onto the second factor. There exists a map that sends this point to an $S$-point $p\in \mathfrak{M}^{(\beta,1)}_{\mathcal{B}_{p}}$ which is obtained by tensoring with $\mathcal{L}_{S}$, i.e $\mathcal{O}_{X}(-n)\boxtimes \mathcal{L}_{S}\xrightarrow{\phi^{\mathcal{L}}} \mathcal{F}\times \pi_{S}^{*}\mathcal{L}_{S}$. Note that tensoring $\mathcal{O}_{X\times S}(-n)$ with $\pi^{*}_{S}\mathcal{L}_{S}$ does not change the fact that $\mathcal{O}_{X\times S}(-n)|_{s\in S}\cong \mathcal{O}_{X}(-n)\boxtimes \mathcal{L}_{S}|_{s\in S}$ fiber by fiber. Moreover, there exists a section map $s: \mathfrak{M}^{(\beta,1)}_{\mathcal{B}_{p}}\rightarrow \mathfrak{M}^{(\beta,1)}_{\mathcal{B}^{\textbf{R}}_{p}}\times [\frac{\operatorname{Spec}(\mathbb{C})}{\mathbb{G}_{m}}]$. Simply take an $S$-point  $[\mathcal{O}_{X}(-n)\boxtimes \mathcal{L}_{S}\rightarrow \mathcal{F}]\in \mathfrak{M}^{(\beta,1)}_{\mathcal{B}_{p}}(S)$  and send it to an $S$-point in $(\mathfrak{M}^{(\beta,1)}_{\mathcal{B}^{\textbf{R}}_{p}}\times [\frac{\operatorname{Spec}(\mathbb{C})}{\mathbb{G}_{m}}])(S)$ by the map $$[\mathcal{O}_{X}(-n)\boxtimes \mathcal{L}_{S}\rightarrow \mathcal{F}]\mapsto([\mathcal{O}_{X\times S}(-n)\rightarrow \mathcal{F}\otimes \pi_{S}^{*}\mathcal{L}_{S}^{-1}], \mathcal{L}_{S}).$$ Note that since $\mathcal{L}_{S}$ is a line bundle over $S$ then it is invertible and hence a section map is well defined globally and $\mathfrak{M}^{(\beta,1)}_{\mathcal{B}_{p}}$ is a $\mathbb{G}_{m}$-gerbe over $\mathfrak{M}^{(\beta,1)}_{\mathcal{B}^{\textbf{R}}_{p}}$.

Now let $r>1$. It is left to show that there exists a map from $\mathfrak{M}^{(\beta,r)}_{\mathcal{B}^{\textbf{R}}_{p}}\times \left[\frac{\operatorname{Spec}(\mathbb{C})}{\operatorname{GL}_{r}(\mathbb{C})}\right] $ to $\mathfrak{M}^{(\beta,r)}_{\mathcal{B}_{p}}$ and this map does not have an inverse (section map) globally. When $r>1$, there exists a forgetful map $\pi:\mathfrak{M}^{(\beta,r)}_{\mathcal{B}^{\textbf{R}}_{p}}\rightarrow \mathfrak{M}^{(\beta,r)}_{\mathcal{B}_{p}}$  which over the $S$-points takes $[\mathcal{O}_{X}(-n)\boxtimes \mathcal{O}_{S}^{\oplus r}\to \mathcal{F}]$ (look at Definition \ref{Bp-R-sflat}) to $[M\boxtimes \mathcal{O}_{X}(-n)\to \mathcal{F}]$ where $M$ is defined in Definition \ref{Bp-sflat}. 

Now start from the family of pairs $M\boxtimes \mathcal{O}_{X}(-n)\to \mathcal{F}$ over $X\times S$. Then, consider the frame bundle $P\rightarrow S$ which is given as a cover of $S$ in the flat topology and observe that, since $M$ is canonically trivialized over $P$, then via the map $u:X\times P\to X\times S$ the pull back of the family, $M\boxtimes \mathcal{O}_{X}(-n)\to \mathcal{F}$, induces a family $$\mathcal{O}_{P}^{\oplus r}\boxtimes \mathcal{O}_{X}(-n)\to u^*{\mathcal{F}}$$ of rigidified objects over $P$ which is seen to induce a section map as in part (3) of the theorem, defined locally but not globally. Note that when $r=1$, as we saw in above $\mathcal{L}_{S}$, is still a locally, but not globally, trivial bundle. However, there we used the invertibility of $\mathcal{L}_{S}$ to get a global section of $GL_{1}(\mathbb{C})$-torsor $\mathfrak{M}^{(\beta,1)}_{\mathcal{B}^{\textbf{R}}_{p}}$ over $\mathfrak{M}^{(\beta,1)}_{\mathcal{B}_{p}}$.  
\end{proof}
\begin{defn}
Define $\mathfrak{M}^{(\beta,r)}_{\mathcal{B}_{p},ss}(\tilde{\tau})$ and $\mathfrak{M}^{(\beta,r)}_{\mathcal{B}_{p},ss}(\tau^{\bullet})$ as substacks of $\mathfrak{M}^{(\beta,r)}_{\mathcal{B}_{p}}$ parameterizing $\tilde{\tau}$-semistable and $\tau^\bullet$-semistable objects in $\mathcal{B}_{p}$ respectively. Since $\mathfrak{M}^{(\beta,r)}_{\mathcal{B}_{p}}$ is of finite type by \cite[Lemma 13.2]{a30} , then $\mathfrak{M}^{(\beta,r)}_{\mathcal{B}_{p},ss}(\tilde{\tau})$ and $\mathfrak{M}^{(\beta,r)}_{\mathcal{B}_{p},ss}(\tau^{\bullet})$ are of finite type for all $(\beta,r)\in \mathcal{C}(\mathcal{B}_{p})$. 
\end{defn}

Now we describe some of the properties of $\mathfrak{M}^{(\beta,r)}_{\mathcal{B}_{p},ss}(\tau^{\bullet})$ in the following proposition.
\begin{prop}\label{stackprop}
The following statements hold true over $\mathfrak{M}^{(\beta,r)}_{\mathcal{B}_{p},ss}(\tau^{\bullet})$.
\begin{enumerate}
\item $\forall (\beta,d)\in C(\mathcal{B}_{p})$ we have natural stack isomorphisms $\mathfrak{M}^{(\beta,0)}_{\mathcal{B}_{p},ss}(\tau^{\bullet})\cong \mathfrak{M}^{\beta}_{ss}(\tau)$ ($\tau$ stands for Gieseker stability condition and $\mathfrak{M}^{\beta}_{ss}(\tau)$ stands for moduli stack of Gieseker semistable coherent sheaves with class $\beta$.) which is obtained by identifying $(F,0,0)$ with $F$. 
\item $\mathfrak{M}^{(0,1)}_{\mathcal{B}_{p},ss}(\tau^{\bullet})\cong [Spec (\mathbb{C})/\mathbb{G}_{m}]$ with the unique point given by $(0,\mathbb{C},0)$. 
\item $\mathfrak{M}^{(\beta,2)}_{\mathcal{B}_{p},ss}(\tau^{\bullet})=\varnothing$ for $\beta\neq 0$. Similarly, $\mathfrak{M}^{(\beta,1)}_{\mathcal{B}_{p},ss}(\tau^{\bullet})=\varnothing$ for $\beta\neq 0$.\\
\item $\mathfrak{M}^{(0,2)}_{\mathcal{B}_{p},ss}(\tau^{\bullet})\cong [\operatorname{Spec}(\mathbb{C})/\operatorname{GL}_{2}(\mathbb{C})]$ with the unique point given by $(0,\mathbb{C}^{2},0)$.
\end{enumerate}
\end{prop}
\begin{proof} The parts (1) and (2) of Proposition \ref{stackprop} are proved in \cite[Proposition 15.6]{a30}. We start by proving part (3). We know that every object $[(F,V,\phi)]=(\beta,2)$ fits in a short exact sequence $$0\rightarrow (F,0,0)\rightarrow (F,V,\phi)\rightarrow (0,V,0)\rightarrow 0,$$ here $[(F,0,0)]=(\beta,0)$ and $[(0,V,0)]=(0,2)$. By Definition \ref{weak} $\tau^{\bullet}(F,0,0)=0>\tau^{\bullet}(0,V,0)=-1$ therefore $(F,0,0)$ $\tau^{\bullet}$-destabilizes $(F,V,\phi)$ for all $[(F,V,\phi)]=(\beta,2)$ and this finishes the proof of part (3). The proof of second part of (3) follows the same strategy and is given in \cite[Proposition 13.6]{a30}. \\
Now we prove part (4); Note that $(0,\mathbb{C}^{2},0)$ is a unique point in $\mathfrak{M}^{(0,2)}_{\mathcal{B}_{p},ss}(\tau^{\bullet})$ which is made of two copies of  $(0,\mathbb{C},0)$ which is the unique object in $\mathfrak{M}^{(0,1)}_{\mathcal{B}_{p},ss}(\tau^{\bullet})$. Moreover, the only nonzero sub-object that can destabilize $(0,\mathbb{C}^{2},0)$ is $(0,\mathbb{C},0)$. There exists a short exact sequence:
\begin{equation}
0\rightarrow (0,\mathbb{C},0)\rightarrow (0,\mathbb{C}^{2},0)\rightarrow (0,\mathbb{C},0)\rightarrow 0.
\end{equation} 
It is easily seen that $\tau^{\bullet}(0,\mathbb{C},0)=\tau^{\bullet}(0,\mathbb{C}^{2},0)=-1$ and therefore the sub-object $(0,\mathbb{C},0)$ does not destabilize $(0,\mathbb{C}^{2},0)$ and $(0,\mathbb{C}^{2},0)$ is weak $\tau^{\bullet}$-semistable. Since the automorphisms of $(0,\mathbb{C}^{2},0)$ are given by $\operatorname{GL}_{2}(\mathbb{C})$ then $\mathfrak{M}^{(0,2)}_{\mathcal{B}_{p},ss}(\tau^{\bullet})\cong [\operatorname{Spec}(\mathbb{C})\slash\operatorname{GL}_{2}(\mathbb{C})]$.
\end{proof}

\section{Stack function identities in the Ringel-Hall algebra}
We review here some basic facts about stack functions in Ringel-Hall algebras. Let $\mathfrak{M}$ be an Artin $\mathbb{C}$-stack with affine geometric stabilizers.  Consider pairs $(\mathfrak{R}, \rho)$ where $\mathfrak{R}$ is given by a finite type Artin $\mathbb{C}$-stack with affine geometric stabilizers and $\rho:=\mathfrak{R}\rightarrow\mathfrak{M}$ is a 1-morphism. Now define an equivalence relation for such pairs where $(\mathfrak{R}, \rho)$ and $(\mathfrak{R}', \rho')$ are called equivalent if there exists a 1-morphism $\iota: \mathfrak{R}\to \mathfrak{R}'$ such that $\rho'\circ \iota$ and $\rho$ are 2-isomorphic 1-morphisms $\mathfrak{R}\to \mathfrak{M}$. Joyce and Song in \cite[Section 2.2]{a30} define the space of stack functions $\underline{\operatorname{S}\overline{\operatorname{F}}}(\mathfrak{M},\chi,\mathbb{Q})$ as the $\mathbb{Q}$-vector space  generated by the above equivalence classes of pairs $[(\mathfrak{R},\rho)]$ such that the following relations are imposed:

\begin{enumerate}
\item Given a closed substack $(\mathfrak{G},\rho|_{\mathfrak{G}})\subset (\mathfrak{R},\rho)$ we have $$[(\mathfrak{R},\rho)]=[(\mathfrak{G},\rho|_{\mathfrak{G}})]+[(\mathfrak{R}/\mathfrak{G},\rho|_{\mathfrak{R}/\mathfrak{G}})]$$ 
\item Let $\mathfrak{R}$ be a $\mathbb{C}$-stack of finite type with affine geometric stabilizers and let $\mathcal{U}$ denote a quasi-projective $\mathbb{C}$-variety and $\pi_{\mathfrak{R}}:\mathfrak{R}\times U\rightarrow \mathfrak{R}$ the natural projection and $\rho:\mathfrak{R}\rightarrow \mathfrak{M}$ a 1-morphism. Then $[(\mathfrak{R}\times \mathcal{U}, \rho \circ \pi_{\mathfrak{R}})]=\chi([\mathcal{U}])[(\mathfrak{R},\rho)]$.
\item Assume $\mathfrak{R}\cong [X/G]$ where $X$ is a quasiprojective $\mathbb{C}$-variety and $G$ a very special algebraic $\mathbb{C}$-group acting on $X$ with maximal torus $T^{G}$, then we have
\begin{equation*}
[(\mathfrak{R},\rho)]=\sum_{Q\in \mathcal{Q}(G,T^{G})}F(G,T^{G},Q)[([X/Q],\rho\circ \iota^{Q})],
\end{equation*}
where the rational coefficients $F(G,T^{G},Q)$ have a complicated definition explained in \cite[Section 6.2]{a33}. 
Here $\mathcal{Q}(G,T^{G})$ is the set of closed $\mathbb{C}$-subgroups $Q$ of $T^{G}$ such that $Q=T^{G}\cap C_{G}(Q)$ where $C_{G}(Q)=\{g\in G:sg=gs \,\,\text{for all} \,\,s\in Q\}$ and $\iota^{Q}:[X/Q]\rightarrow\mathfrak{R}\cong [X/G]$ is the natural projection 1-morphism. Similarly, one defines $\overline{\operatorname{SF}}(\mathfrak{M},\chi,\mathbb{Q})$ by restricting the 1-morphisms $\rho$ in part 1, 2, 3 to be representable.
\item There exist the notions of multiplication, pullback, pushforward of stack functions in  $\underline{\operatorname{S}\overline{\operatorname{F}}}(\mathfrak{M},\chi,\mathbb{Q})$ and $\overline{\operatorname{SF}}(\mathfrak{M},\chi,\mathbb{Q})$. For further discussions look at (Joyce and Song) \cite{a30} (Definitions 2.6, 2.7) and (Theorem 2.9). 
\item Joyce and Song in \cite{a30} (Section 13.3) define the notion of characteristic stack functions  $\overline{\delta}^{(\beta,d)}_{ss}(\tilde{\tau})\in \overline{\operatorname{SF}}(\mathfrak{M}_{\mathcal{B}_{p}}(\tilde{\tau}),\chi,\mathbb{Q})$ and $\overline{\delta}^{(\beta,d)}_{ss}(\tau^{\bullet})\in \overline{\operatorname{SF}}(\mathfrak{M}_{\mathcal{B}_{p}}(\tau^{\bullet}),\chi,\mathbb{Q})$. Moreover, in the instance where the moduli stack contains strictly semistable objects, the authors define the ``\textit{logarithm}" of the moduli stack by the stack function $\overline{\epsilon}^{(\beta,d)}(\tilde{\tau})$ given as an element of the Hall-algebra of stack functions supported over virtual indecomposables.
\end{enumerate}
To continue we state the wallcrossing formula under change of stability condition from $(\tau^{\bullet}, T^{\bullet}, \leq)$ to $(\tilde{\tau}, \tilde{T}, \leq)$; 
\begin{prop}
\cite[Proposition 13.7]{a30}. For all $(\beta,d)$ in $C(\mathcal{B}_{p})$, the following identity holds in the Ringel-Hall algebra of $\mathcal{B}_{p}$:
\begin{align}\label{eqU2}
&
\bar{\epsilon}^{(\beta,d)}(\tilde{\tau})=
\sum_{n\geq 1}\sum_{\begin{subarray}{1}((\beta_{1},d_{1}),\cdots,(\beta_{n},d_{n}))\in \mathcal{C}(\mathcal{B}_{p})^{n}:\\ (\beta_{1},d_{1})+\cdots+(\beta_{n},d_{n})=(\beta,d)\end{subarray}} U\bigg((\beta_{1},d_{1}),\cdots (\beta_{n},d_{n});\tau^{\bullet},\tilde{\tau}\bigg)\notag\\
&
\cdot \bar{\epsilon}^{(\beta_{1},d_{1})}(\tau^{\bullet})*\cdots \cdots*\bar{\epsilon}^{(\beta_{n},d_{n})}(\tau^{\bullet}).\notag\\
\end{align} 
There are only finitely many choices of $n\geq 1$ as well as $(\beta_{i},d_{i})\in C(\mathcal{B}_{p})$ for which the coefficients $U\bigg((\beta_{1},d_{1}),\cdots (\beta_{n},d_{n});\tau^{\bullet},\tilde{\tau}\bigg)$ do not vanish.
\end{prop}
 Now we recall the definition of the function $U$ in Equation \eqref{eqU2} from \cite[Definition 3.8]{a30};
\begin{defn} \label{SBp}
Let $n\geq 1$ and $$(\beta_{1},d_{1}),\cdots, (\beta_{n},d_{n})\in C(\mathcal{B}_{p}).$$We define a number, $S((\beta_{1},d_{1}),\cdots,(\beta_{n},d_{n});\tau^{\bullet},\tilde{\tau})$ as follows: If for all $i=1,\cdots,n-1$ we have either: 
\begin{enumerate} [(a)]
\item $\tau^{\bullet}(\beta_{i},d_{i})\leq \tau^{\bullet}(\beta_{i+1},d_{i+1})$ and $$\tilde{\tau}((\beta_{1},d_{1})+\cdots+(\beta_{i},d_{i}))>\tilde{\tau}((\beta_{i+1},d_{i+1})+\cdots+(\beta_{n},d_{n})).$$ or\\
\item $\tau^{\bullet}(\beta_{i},d_{i})> \tau^{\bullet}(\beta_{i+1},d_{i+1})$ and $$\tilde{\tau}((\beta_{1},d_{1})+\cdots+(\beta_{i},d_{i}))\leq\tilde{\tau}((\beta_{i+1},d_{i+1})+\cdots+(\beta_{n},d_{n})),$$
\end{enumerate}
then define $S((\beta_{1},d_{1}),\cdots,(\beta_{n},d_{n});\tau^{\bullet},\tilde{\tau})=(-1)^{r}$, where $$r=\#\{i\in \{1,\cdots, n-1\}| (a)\,\, \text{holds}\},$$ otherwise if for all $i=1,\cdots,n-1$, neither (a) nor (b) is true, then set $S=0$.
Given $n\geq 1$ and $(\beta_{1},d_{1}),\cdots,(\beta_{n},d_{n})$ as above, choose two numbers $l$ and $m$ such that $1\leq l\leq m\leq n$. Now for this choice choose numbers $0=a_{0}<a_{1}<\cdots<a_{m}=n$ and $0=b_{0}<b_{1}<\cdots<b_{l}=m$. Given such $m$ and $a_{1},\cdots,a_{m}$, define elements $\theta_{1},\cdots,\theta_{m}\in C(\mathcal{B}_{p})$ by $\theta_{i}=(\beta_{a_{i-1}+1},d_{a_{i-1}+1})+\cdots +(\beta_{a_{i}},d_{a_{i}})$ (to add two pairs just add them coordinate-wise in $C(\mathcal{B}_{p})$). Also given such $l,b_{1},\cdots,b_{l}$ define elements $\gamma_{1},\cdots,\gamma_{l}\in C(\mathcal{B}_{p})$ by $\gamma_{i}=\theta_{b_{i-1}+1}+\cdots \theta_{b_{i}}$. Let $\Lambda$ denote the set of choices $(l,m,a_{1},\cdots,a_{m},b_{1},\cdots,b_{l})$ for which the two following conditions are satisfied:
\begin{enumerate}
\item $\tau^{\bullet}(\theta_{i})=\tau^{\bullet}(\beta_{j},d_{j})$ for $i=1,\cdots, m$ and $a_{i-1}<j\leq a_{i}$.\\
\item $\tilde{\tau}(\gamma_{i})=\tilde{\tau}(\beta,d)$ for $i=1,\cdots l$ (here $\beta=\displaystyle{\sum}_{i=1}^{n}\beta_{i}$ and $d=\displaystyle{\sum}_{i=1}^{n}d_{i}$). 
\end{enumerate}
Now define:
\begin{align} \label{eqU3}
&
U\bigg((\beta_{1},d_{1}),\cdots,(\beta_{n},d_{n});\tau^{\bullet},\tilde{\tau}\bigg)=\notag\\
&
\sum_{\Lambda}\frac{(-1)^{l-1}}{l}\prod^{l}_{i=1}S(\theta_{b_{i-1}+1},\theta_{b_{i-1}+2},\cdots, \theta_{b_{i}};\tau^{\bullet},\tilde{\tau})\cdot \prod^{m}_{i=1}\frac{1}{(a_{i}-a_{i-1})!}.
\end{align}
\end{defn}
\section{Wallcrossing computations for objects with class $(\beta,2)$ in $\mathcal{B}_{p}$}
Our main goal here is to compute the wall-crossing identity for the invariants of objects of type $(\beta,2)$ in $\mathcal{B}_{p}$ by changing the weak stability condition from $\tau^{\bullet}$ to $\tilde{\tau}$. First we compute the Hall algebra element $\bar{\epsilon}^{(\beta,2)}(\tilde{\tau})$ on the left hand side of \eqref{eqU2} with respect to the product of $\bar{\epsilon}^{(\beta_{i},d_{i})}(\tau^{\bullet})$'s appearing on the right hand side. Note that the sum on the right hand side of \eqref{eqU2} is over all possible decompositions of the class $(\beta,2)$ into irreducible classes $(\beta_{i},d_{i})$. We will also need to compute the combinatorial coefficient $U\bigg((\beta_{1},d_{1}),\cdots,(\beta_{n},d_{n});\tau^{\bullet},\tilde{\tau}\bigg)$ for our calculation.

We decompose the class $(\beta,2)$ into irreducible classes. First, decompose $d=2$ and then decompose $\beta$. The only two possible ways to break $d=2$ is to write $2=2+0$ and $2=1+1$. Now for each choice of decomposition of $d$,  one decomposes $\beta$ into irreducible classes $\beta_{i}$. For example for the case $2=2+0$, the decomposition of $\beta$ produces elements in  $C(\mathcal{B}_{p})$ of type $(\beta_{1},d_{1}),\cdots, (\beta_{n},d_{n})$, where $\beta_{1}+\cdots+\beta_{n}=\beta$ and $d_{1}+\cdots+d_{n}=2$, hence there exists a tuple in this sequence which is of type $(\beta_{i},2)$ and the remaining objects are of type $(\beta_{j},0)$. Now use Proposition \ref{stackprop} and note that $\mathfrak{M}^{(\beta_{i},2)}_{\mathcal{B}_{p},ss}(\tau^{\bullet})=\varnothing$, unless $\beta_{i}=0$. Hence, the corresponding sequence of numerical classes is given as $$(\beta_{1},0),\cdots,(0,2),\cdots,(\beta_{n},0).$$ Similarly for the decomposition of type $2=1+1$, and using Proposition \ref{stackprop}, one obtains elements of type $$(\beta_{1},0),\cdots,(\beta_{k-1},0),(0,1),\cdots,(\beta_{t-1},0),(0,1),\cdots,(\beta_{n},0)\,\,\, \text{for}\,\,\, 1\leq k< t\leq n,$$where the two $(0,1)$ elements (one in the $k$'th location and the other in $t$'th location), float in between the elements of type $(\beta_{i},0)$ in the sequence.
In order to ease the bookkeeping, we use a reparametrization of $(\beta_{i},d_{i})$ which is consistent with the work of Joyce and Song. For a decomposition $2=2+0$ define $(\psi_{i},d_{i})=(\beta_{i},0)$ for $i\leq k-1$,  and $(\psi_{i},d_{i})=(\beta_{i+1},0)$ for $i\geq k$. For decomposition of type $2=1+1$ define $(\psi_{i},d_{i})=(\beta_{i},0)$ for $i\leq k-1$, $(\psi_{i},d_{i})=(\beta_{i+1},0)$ for $k\leq i \leq t-2$ and $(\psi_{i},d_{i})=(\beta_{i+2},0)$ for $i\geq t-1$.
\begin{defn}
\begin{enumerate}
\item Fix some $k$ such that $1\leq k \leq n$ ($k$ shows the location of $(0,2)$ element). Given a sequence of numerical classes in $C(\mathcal{B}_{p})$:$$(\psi_{1},0),\cdots (\psi_{k-1},0),(0,2),(\psi_{k},0), \cdots, (\psi_{n-1},0),$$define $$U_{k}=U\bigg((\psi_{1},0),\cdots (\psi_{k-1},0),(0,2),(\psi_{k},0), \cdots, (\psi_{n-1},0);\tau^{\bullet},\tilde{\tau}\bigg).$$
\item Similarly, fix some $k,t$ such that $1\leq k< t\leq n$ ($k,t$ show the location of first and second $(0,1)$ elements in the sequence). Given a sequence $$(\psi_{1},0),\cdots,(\psi_{k-1},0),(0,1),(\psi_{k},0),\cdots,(\psi_{t-2},0),(0,1),(\psi_{t-1},0)\cdots, (\psi_{n-2},0)$$define 
\begin{align*}
&
U_{k,t}=
U\bigg((\psi_{1},0),\cdots,(\psi_{k-1},0),(0,1),(\psi_{k},0),\cdots,(\psi_{t-2},0),(0,1),(\psi_{t-1},0),\cdots\notag\\
&
\cdots, (\psi_{n-2},0);\tau^{\bullet},\tilde{\tau}\bigg)
\end{align*}
\end{enumerate}
\end{defn}
Now Equation \eqref{eqU2} for the case of $(\beta,2)$ is written as: 
\begin{align}\label{b,2}
&
\bar{\epsilon}^{(\beta,2)}(\tilde{\tau})=\Bigg[\sum_{\begin{subarray}{1}1\leq k \leq n\\ \psi_{1},\cdots, \psi_{n-1}\in C(\mathcal{A}_{p})\\ \psi_{1}+\cdots+\psi_{n-1}=\beta \end{subarray}} U_{k}\cdot\bar{\epsilon}^{(\psi_{1},0)}(\tau^{\bullet})*\cdots *\bar{\epsilon}^{(\psi_{k-1},0)}(\tau^{\bullet})*\bar{\epsilon}^{(0,2)}(\tau^{\bullet})*\bar{\epsilon}^{(\psi_{k},0)}(\tau^{\bullet})*\cdots*\notag\\
&
\bar{\epsilon}^{(\psi_{n-1},0)}(\tau^{\bullet})\Bigg]+\Bigg[\sum_{\begin{subarray}{1}1\leq k< t \leq n\\ \psi_{1},\cdots, \psi_{n-2}\in C(\mathcal{A}_{p})\\ \psi_{1}+\cdots+\psi_{n-2}=\beta\end{subarray}} U_{k,t}\cdot \bar{\epsilon}^{(\psi_{1},0)}(\tau^{\bullet})*\cdots *\bar{\epsilon}^{(\psi_{k-1},0)}(\tau^{\bullet})*\bar{\epsilon}^{(0,1)}(\tau^{\bullet})*\bar{\epsilon}^{(\psi_{k},0)}(\tau^{\bullet})\notag\\
&
*\cdots*\bar{\epsilon}^{(\psi_{t-2},0)}(\tau^{\bullet})*\bar{\epsilon}^{(0,1)}(\tau^{\bullet})*\bar{\epsilon}^{(\psi_{t-1},0)}(\tau^{\bullet})*\cdots *\bar{\epsilon}^{(\psi_{n-2},0)}(\tau^{\bullet})\Bigg]\notag\\
\end{align}
Let $\textbf{E}_{1}$ and $\textbf{E}_{2}$ respectively denote the first and second brackets on the right hand side of \eqref{b,2}. 
\subsection{Computation of $\textbf{E}_{1}$}
By \eqref{b,2} and \eqref{eqU3} $U_{k}$ is given by:
\begin{align}\label{UA1}
&
U_{k}=U\bigg((\psi_{1},0),\cdots (\psi_{k-1},0),(0,2),(\psi_{k},0), \cdots, (\psi_{n-1},0);\tau^{\bullet},\tilde{\tau}\bigg)=\notag\\
&
\sum_{\Lambda} \frac{(-1)^{l-1}}{l}\cdot \prod^{l}_{i=1}S_{\textbf{E}_{1}}(\theta_{b_{i-1}+1},\theta_{b_{i-1}+2},\cdots \theta_{b_{i}};\tau^{\bullet},\tilde{\tau})\cdot\prod^{m}_{i=1}\frac{1}{(a_{i}-a_{i-1})!}.\notag\\
\end{align}
\begin{remark}
Here by notation $S_{\textbf{E}_{1}}(\theta_{b_{i-1}+1},\theta_{b_{i-1}+2},\cdots \theta_{b_{i}};\tau^{\bullet},\tilde{\tau})$ in Equation \eqref{UA1} we mean the $S$ function defined in Definition \ref{SBp} for the specific configuration of elements appearing in $\textbf{E}_{1}$. Similarly later in Equation \eqref{b,k,m}, we will use $S_{\textbf{E}_{2}}(\theta_{b_{i-1}+1},\theta_{b_{i-1}+2},\cdots \theta_{b_{i}};\tau^{\bullet},\tilde{\tau})$ to denote the $S$ function for configuration of elements appearing in $\textbf{E}_{2}$. 
\end{remark}
Here we compute $U_{k}$. Apply Definition \ref{SBp} and obtain the following conditions:
\begin{enumerate}
\item In order to have $\tilde{\tau}(\gamma_{i})=\tilde{\tau}(\beta,2)$ for all $i=1,\cdots, l$ one should set  $l=1$, \cite{a30} (Proposition 15.8). Therefore the set $\Lambda$ reduces to the set of choices of $m$ where $1\leq m\leq n$. 
\item It is clear that the only way that $\tau^{\bullet}(\theta_{i})=\tau^{\bullet}(\beta_{j},d_{j})$ for $i=1,\cdots, m$ and $a_{i-1}<j\leq a_{i}$ is that there exists some $p\in \{1,\cdots, m\}$ where $a_{p-1}=k-1$ and $a_{p}=k$ ($k=$location of $(0,2)$).
\end{enumerate}
In \eqref{UA1} $\tau^{\bullet}(\theta_{i})=0$ for $i<p$ and $\tau^{\bullet}(\theta_{p})=-1$ and $\tau^{\bullet}(\theta_{i})=0$ for $i> p$, therefore the following hold true:
\begin{enumerate}
\item $\tau^{\bullet}(\theta_{i})=\tau^{\bullet}(\theta_{i+1})=0$ and $\tilde{\tau}(\theta_{1}+\cdots+\theta_{i})\ngtr\tilde{\tau}(\theta_{i+1}+\cdots+\theta_{n})$ for $i< p-1$\\
\item $0=\tau^{\bullet}(\theta_{i})> \tau^{\bullet}(\theta_{i+1})=-1$ and $0=\tilde{\tau}(\theta_{1}+\cdots+\theta_{i})\leq \tilde{\tau}(\theta_{i+1}+\cdots+\theta_{n})=1$ for $i=p-1$\\
\item $\tau^{\bullet}(\theta_{i})\leq\tau^{\bullet}(\theta_{i+1})$ and $\tilde{\tau}(\theta_{1}+\cdots+\theta_{i})>\tilde{\tau}(\theta_{i+1}+\cdots+\theta_{n})$ for $i\geq p$
\end{enumerate}
From this analysis one concludes that in \eqref{UA1} for $i<p-1$ neither condition (a) nor (b) are satisfied, for $i=p-1$ condition (b) is satisfied, and for $i\geq p$ condition (a) is satisfied (this implies $p=1$ or $p=2$). 
Moreover, $p=1$ when $k=1$ and $p>1$ when $k>1$ and $S_{\textbf{E}_{1}}=0$ for $p>2$. By the above computations when $p=1$ we have$$(U_{k})|_{p=1}=\sum_{\substack{1\leq m\leq n,\\1=a_{1}<a_{2}<\cdots<a_{m}}}(-1)^{m-1}\cdot \prod_{i=2}^{m}\frac{1}{(a_{i}-a_{i-1})!},$$and for $p=2$ and each fixed $k$ such that $1<k\leq n$ we have$$(U_{k})|_{p=2}=\frac{1}{(k-1)!}\sum_{\begin{subarray}{1}1\leq m\leq n\\k=a_{2}<a_{3}<\cdots<a_{m}=n\end{subarray}}(-1)^{m-2}\cdot \prod_{i=3}^{m}\frac{1}{(a_{i}-a_{i-1})!}.$$ 
Now we can compute $\textbf{E}_{1}$ as follows:
\begin{align}\label{A}
\textbf{E}_{1}=&\sum_{\substack{1\leq m\leq n,\\1=a_{1}<a_{2}<\cdots<a_{m}=n\\\\ \psi_{1},\cdots, \psi_{n-1}\in C(\mathcal{A}_{p})\\ \psi_{1}+\cdots+\psi_{n-1}=\beta}}(-1)^{m-1}\cdot \prod_{i=2}^{m}\frac{1}{(a_{i}-a_{i-1})!}\cdot\bar{\epsilon}^{(0,2)}*\bar{\epsilon}^{(\psi_{2},0)}*\cdots*\bar{\epsilon}^{(\psi_{n-1},0)}\notag\\
&
+\sum_{\begin{subarray}{1}1<k\leq n\\\\ \psi_{1},\cdots, \psi_{n-1}\in C(\mathcal{A}_{p})\\ \psi_{1}+\cdots+\psi_{n-1}=\beta\end{subarray}}\frac{1}{(k-1)!}\cdot \sum_{1\leq m\leq n,k=a_{2}<a_{3}<\cdots<a_{m}=n}(-1)^{m-2}\cdot \prod_{i=3}^{m}\frac{1}{(a_{i}-a_{i-1})!}\notag\\
&
\cdot\bar{\epsilon}^{(\psi_{1},0)}*\cdots *\bar{\epsilon}^{(\psi_{k-1},0)}*\bar{\epsilon}^{(0,2)}*\bar{\epsilon}^{(\psi_{k},0)}*\cdots*\bar{\epsilon}^{(\psi_{n-1},0)}\notag\\
\end{align}
\subsection{Computation of $\textbf{E}_{2}$}\label{sec9}
By Equations \eqref{b,2} and \eqref{eqU3}, $U_{k,t}$ is given as 
\begin{align}\label{b,k,m}
&
 U_{k,t}=\sum_{1\leq l \leq m\leq n} \frac{(-1)^{l-1}}{l}\cdot \prod^{l}_{i=1}S_{\textbf{E}_{2}}(\theta_{b_{i-1}+1},\theta_{b_{i-1}+2},\cdots \theta_{b_{i}};\tau^{\bullet},\tilde{\tau})\cdot\prod^{m}_{i=1}\frac{1}{(a_{i}-a_{i-1})!}\notag\\
\end{align}
\begin{lemma}\label{UB}
Consider the notation in Equation \eqref{b,k,m}. Then $U_{k,t}=0.$
\end{lemma}
\begin{proof}
In order to evaluate $U_{k,t}$ we need to compute the combinatorial coefficients $S_{\textbf{E}_{2}}(\theta_{b_{i-1}+1},\theta_{b_{i-1}+2},\cdots \theta_{b_{i}};\tau^{\bullet},\tilde{\tau})$ (in short $S_{\textbf{E}_{2}}$) appearing on the right hand side of Equation \eqref{b,k,m}. To compute $S_{\textbf{E}_{2}}$ we divide our analysis into three separate combinatorial cases (Case 1, Case 2 and Case 3) based on how the  $(0,1)$ elements are located in the sequence of $(\psi_{i},d_{i})$'s. We denote the contributions to $U_{k,t}$ in each case by $U^{1}_{k,t}, U^{2}_{k,t}$ and $U^{3}_{k,t}$. Moreover, for $i=1,2,3$ we denote by $S^{i}_{\textbf{E}_{2}}$ the value of the function $S_{\textbf{E}_{2}}$ corresponding to $U^{i}_{k,t}$. Note that by our construction $U_{k,t}=U^{1}_{k,t}+U^{2}_{k,t}+U^{3}_{k,t}.$
\subsubsection{Computations in Case 1:}
Case 1 represents the configurations, where the two $(0,1)$ elements occur adjacent to each other. In this case it is seen that $U^{1}_{k,t}=U_{k,k+1}$ (the two $(0,1)$ elements are adjacent, therefore $t=k+1$). Now we need to choose and distribute $a_{i}$ in order to obtain equation \eqref{eqU3}. The following diagrams describe the two possible distribution types for $a_{i}$, we denote them by Case 1-a and Case 1-b. In Case 1-a we have $a_{1}=k-1, a_{2}=k, a_{3}=k+1$ and $a_{4}$ can be chosen freely as long as $a_{4}\geq k+2$.  

\ifx\JPicScale\undefined\def\JPicScale{.9}\fi
\unitlength \JPicScale mm
\begin{picture}(128.05,57)(-5,21)
\linethickness{0.3mm}
\put(0,45){\line(1,0){120}}
\put(0,45){\makebox(0,0)[cc]{$\bullet$}}

\put(10,45){\makebox(0,0)[cc]{$\bullet$}}

\put(20,45){\makebox(0,0)[cc]{$\bullet$}}

\put(30,45){\makebox(0,0)[cc]{$\bullet$}}

\put(40,45){\makebox(0,0)[cc]{$\bullet$}}

\put(50,45){\makebox(0,0)[cc]{$\bullet$}}

\put(60,45){\makebox(0,0)[cc]{$\bullet$}}

\put(70,45){\makebox(0,0)[cc]{$\bullet$}}

\put(80,45){\makebox(0,0)[cc]{$\bullet$}}

\put(90,45){\makebox(0,0)[cc]{$\bullet$}}

\put(100,45){\makebox(0,0)[cc]{$\bullet$}}

\put(110,45){\makebox(0,0)[cc]{$\bullet$}}

\put(120,45){\makebox(0,0)[cc]{$\bullet$}}

\linethickness{0.3mm}
\put(20,45){\line(0,1){10}}
\linethickness{0.3mm}
\put(50,45){\line(0,1){10}}
\linethickness{0.3mm}
\put(120,45){\line(0,1){10}}
\put(10,25){\makebox(0,0)[cc]{$\theta_{1}$}}

\put(30,25){\makebox(0,0)[cc]{$\theta_{2}$}}

\linethickness{0.3mm}
\put(40,45){\line(0,1){20}}
\put(40,40){\makebox(0,0)[cc]{}}

\linethickness{0.3mm}
\put(30,45){\line(0,1){20}}
\put(30,70){\makebox(0,0)[cc]{$(0,1)$}}

\put(40,70){\makebox(0,0)[cc]{$(0,1)$}}

\linethickness{0.3mm}
\multiput(0,35)(0.1,-0.13){3}{\line(0,-1){0.13}}
\multiput(0.31,34.6)(0.11,-0.13){3}{\line(0,-1){0.13}}
\multiput(0.64,34.22)(0.11,-0.12){3}{\line(0,-1){0.12}}
\multiput(0.98,33.85)(0.12,-0.12){3}{\line(0,-1){0.12}}
\multiput(1.33,33.49)(0.12,-0.11){3}{\line(1,0){0.12}}
\multiput(1.7,33.15)(0.13,-0.11){3}{\line(1,0){0.13}}
\multiput(2.09,32.82)(0.13,-0.1){3}{\line(1,0){0.13}}
\multiput(2.48,32.51)(0.2,-0.15){2}{\line(1,0){0.2}}
\multiput(2.89,32.22)(0.21,-0.14){2}{\line(1,0){0.21}}
\multiput(3.31,31.94)(0.22,-0.13){2}{\line(1,0){0.22}}
\multiput(3.74,31.68)(0.22,-0.12){2}{\line(1,0){0.22}}
\multiput(4.19,31.43)(0.23,-0.11){2}{\line(1,0){0.23}}
\multiput(4.64,31.21)(0.23,-0.1){2}{\line(1,0){0.23}}
\multiput(5.1,31)(0.23,-0.09){2}{\line(1,0){0.23}}
\multiput(5.56,30.81)(0.47,-0.17){1}{\line(1,0){0.47}}
\multiput(6.04,30.64)(0.48,-0.15){1}{\line(1,0){0.48}}
\multiput(6.52,30.49)(0.49,-0.13){1}{\line(1,0){0.49}}
\multiput(7.01,30.36)(0.49,-0.11){1}{\line(1,0){0.49}}
\multiput(7.5,30.25)(0.5,-0.09){1}{\line(1,0){0.5}}
\multiput(7.99,30.16)(0.5,-0.07){1}{\line(1,0){0.5}}
\multiput(8.49,30.09)(0.5,-0.05){1}{\line(1,0){0.5}}
\multiput(8.99,30.04)(0.5,-0.03){1}{\line(1,0){0.5}}
\multiput(9.5,30.01)(0.5,-0.01){1}{\line(1,0){0.5}}
\multiput(10,30)(0.5,0.01){1}{\line(1,0){0.5}}
\multiput(10.5,30.01)(0.5,0.03){1}{\line(1,0){0.5}}
\multiput(11.01,30.04)(0.5,0.05){1}{\line(1,0){0.5}}
\multiput(11.51,30.09)(0.5,0.07){1}{\line(1,0){0.5}}
\multiput(12.01,30.16)(0.5,0.09){1}{\line(1,0){0.5}}
\multiput(12.5,30.25)(0.49,0.11){1}{\line(1,0){0.49}}
\multiput(12.99,30.36)(0.49,0.13){1}{\line(1,0){0.49}}
\multiput(13.48,30.49)(0.48,0.15){1}{\line(1,0){0.48}}
\multiput(13.96,30.64)(0.47,0.17){1}{\line(1,0){0.47}}
\multiput(14.44,30.81)(0.23,0.09){2}{\line(1,0){0.23}}
\multiput(14.9,31)(0.23,0.1){2}{\line(1,0){0.23}}
\multiput(15.36,31.21)(0.23,0.11){2}{\line(1,0){0.23}}
\multiput(15.81,31.43)(0.22,0.12){2}{\line(1,0){0.22}}
\multiput(16.26,31.68)(0.22,0.13){2}{\line(1,0){0.22}}
\multiput(16.69,31.94)(0.21,0.14){2}{\line(1,0){0.21}}
\multiput(17.11,32.22)(0.2,0.15){2}{\line(1,0){0.2}}
\multiput(17.52,32.51)(0.13,0.1){3}{\line(1,0){0.13}}
\multiput(17.91,32.82)(0.13,0.11){3}{\line(1,0){0.13}}
\multiput(18.3,33.15)(0.12,0.11){3}{\line(1,0){0.12}}
\multiput(18.67,33.49)(0.12,0.12){3}{\line(0,1){0.12}}
\multiput(19.02,33.85)(0.11,0.12){3}{\line(0,1){0.12}}
\multiput(19.36,34.22)(0.11,0.13){3}{\line(0,1){0.13}}
\multiput(19.69,34.6)(0.1,0.13){3}{\line(0,1){0.13}}

\linethickness{0.3mm}
\multiput(100,30)(0.49,-0.1){1}{\line(1,0){0.49}}
\multiput(100.49,29.9)(0.49,-0.09){1}{\line(1,0){0.49}}
\multiput(100.97,29.81)(0.49,-0.08){1}{\line(1,0){0.49}}
\multiput(101.47,29.72)(0.49,-0.07){1}{\line(1,0){0.49}}
\multiput(101.96,29.65)(0.49,-0.06){1}{\line(1,0){0.49}}
\multiput(102.45,29.59)(0.49,-0.05){1}{\line(1,0){0.49}}
\multiput(102.94,29.54)(0.5,-0.04){1}{\line(1,0){0.5}}
\multiput(103.44,29.5)(0.5,-0.03){1}{\line(1,0){0.5}}
\multiput(103.94,29.48)(0.5,-0.02){1}{\line(1,0){0.5}}
\multiput(104.43,29.46)(0.5,-0.01){1}{\line(1,0){0.5}}
\multiput(104.93,29.45)(0.5,0){1}{\line(1,0){0.5}}
\multiput(105.43,29.46)(0.5,0.01){1}{\line(1,0){0.5}}
\multiput(105.92,29.47)(0.5,0.03){1}{\line(1,0){0.5}}
\multiput(106.42,29.49)(0.5,0.04){1}{\line(1,0){0.5}}
\multiput(106.92,29.53)(0.49,0.05){1}{\line(1,0){0.49}}
\multiput(107.41,29.58)(0.49,0.06){1}{\line(1,0){0.49}}
\multiput(107.9,29.63)(0.49,0.07){1}{\line(1,0){0.49}}
\multiput(108.4,29.7)(0.49,0.08){1}{\line(1,0){0.49}}
\multiput(108.89,29.78)(0.49,0.09){1}{\line(1,0){0.49}}
\multiput(109.38,29.87)(0.49,0.1){1}{\line(1,0){0.49}}
\multiput(109.86,29.97)(0.48,0.11){1}{\line(1,0){0.48}}
\multiput(110.35,30.08)(0.48,0.12){1}{\line(1,0){0.48}}
\multiput(110.83,30.2)(0.48,0.13){1}{\line(1,0){0.48}}
\multiput(111.31,30.33)(0.48,0.14){1}{\line(1,0){0.48}}
\multiput(111.79,30.47)(0.47,0.15){1}{\line(1,0){0.47}}
\multiput(112.26,30.62)(0.47,0.16){1}{\line(1,0){0.47}}
\multiput(112.73,30.79)(0.47,0.17){1}{\line(1,0){0.47}}
\multiput(113.2,30.96)(0.23,0.09){2}{\line(1,0){0.23}}
\multiput(113.66,31.14)(0.23,0.1){2}{\line(1,0){0.23}}
\multiput(114.12,31.33)(0.23,0.1){2}{\line(1,0){0.23}}
\multiput(114.57,31.53)(0.22,0.11){2}{\line(1,0){0.22}}
\multiput(115.02,31.74)(0.22,0.11){2}{\line(1,0){0.22}}
\multiput(115.47,31.97)(0.22,0.12){2}{\line(1,0){0.22}}
\multiput(115.91,32.2)(0.22,0.12){2}{\line(1,0){0.22}}
\multiput(116.34,32.44)(0.22,0.12){2}{\line(1,0){0.22}}
\multiput(116.77,32.68)(0.21,0.13){2}{\line(1,0){0.21}}
\multiput(117.2,32.94)(0.21,0.13){2}{\line(1,0){0.21}}
\multiput(117.62,33.21)(0.21,0.14){2}{\line(1,0){0.21}}
\multiput(118.03,33.49)(0.2,0.14){2}{\line(1,0){0.2}}
\multiput(118.44,33.77)(0.2,0.15){2}{\line(1,0){0.2}}
\multiput(118.84,34.07)(0.13,0.1){3}{\line(1,0){0.13}}
\multiput(119.23,34.37)(0.13,0.1){3}{\line(1,0){0.13}}
\multiput(119.62,34.68)(0.13,0.11){3}{\line(1,0){0.13}}

\linethickness{0.3mm}
\multiput(60,35)(0.21,-0.13){2}{\line(1,0){0.21}}
\multiput(60.42,34.74)(0.21,-0.13){2}{\line(1,0){0.21}}
\multiput(60.85,34.48)(0.22,-0.13){2}{\line(1,0){0.22}}
\multiput(61.28,34.22)(0.22,-0.12){2}{\line(1,0){0.22}}
\multiput(61.72,33.98)(0.22,-0.12){2}{\line(1,0){0.22}}
\multiput(62.16,33.74)(0.22,-0.12){2}{\line(1,0){0.22}}
\multiput(62.6,33.51)(0.22,-0.11){2}{\line(1,0){0.22}}
\multiput(63.04,33.28)(0.22,-0.11){2}{\line(1,0){0.22}}
\multiput(63.49,33.07)(0.23,-0.11){2}{\line(1,0){0.23}}
\multiput(63.95,32.86)(0.23,-0.1){2}{\line(1,0){0.23}}
\multiput(64.4,32.65)(0.23,-0.1){2}{\line(1,0){0.23}}
\multiput(64.86,32.45)(0.23,-0.1){2}{\line(1,0){0.23}}
\multiput(65.32,32.26)(0.23,-0.09){2}{\line(1,0){0.23}}
\multiput(65.79,32.08)(0.47,-0.18){1}{\line(1,0){0.47}}
\multiput(66.25,31.91)(0.47,-0.17){1}{\line(1,0){0.47}}
\multiput(66.72,31.74)(0.47,-0.16){1}{\line(1,0){0.47}}
\multiput(67.2,31.58)(0.48,-0.15){1}{\line(1,0){0.48}}
\multiput(67.67,31.42)(0.48,-0.15){1}{\line(1,0){0.48}}
\multiput(68.15,31.27)(0.48,-0.14){1}{\line(1,0){0.48}}
\multiput(68.63,31.13)(0.48,-0.13){1}{\line(1,0){0.48}}
\multiput(69.11,31)(0.48,-0.12){1}{\line(1,0){0.48}}
\multiput(69.59,30.88)(0.49,-0.12){1}{\line(1,0){0.49}}
\multiput(70.08,30.76)(0.49,-0.11){1}{\line(1,0){0.49}}
\multiput(70.57,30.65)(0.49,-0.1){1}{\line(1,0){0.49}}
\multiput(71.06,30.55)(0.49,-0.09){1}{\line(1,0){0.49}}
\multiput(71.55,30.45)(0.49,-0.09){1}{\line(1,0){0.49}}
\multiput(72.04,30.36)(0.49,-0.08){1}{\line(1,0){0.49}}
\multiput(72.53,30.29)(0.49,-0.07){1}{\line(1,0){0.49}}
\multiput(73.02,30.21)(0.5,-0.06){1}{\line(1,0){0.5}}
\multiput(73.52,30.15)(0.5,-0.06){1}{\line(1,0){0.5}}
\multiput(74.02,30.09)(0.5,-0.05){1}{\line(1,0){0.5}}
\multiput(74.51,30.04)(0.5,-0.04){1}{\line(1,0){0.5}}
\multiput(75.01,30)(0.5,-0.03){1}{\line(1,0){0.5}}
\multiput(75.51,29.96)(0.5,-0.03){1}{\line(1,0){0.5}}
\multiput(76.01,29.94)(0.5,-0.02){1}{\line(1,0){0.5}}
\multiput(76.51,29.92)(0.5,-0.01){1}{\line(1,0){0.5}}
\multiput(77.01,29.91)(0.5,-0){1}{\line(1,0){0.5}}
\multiput(77.51,29.9)(0.5,0){1}{\line(1,0){0.5}}
\multiput(78,29.91)(0.5,0.01){1}{\line(1,0){0.5}}
\multiput(78.5,29.92)(0.5,0.02){1}{\line(1,0){0.5}}
\multiput(79,29.94)(0.5,0.03){1}{\line(1,0){0.5}}
\multiput(79.5,29.97)(0.5,0.03){1}{\line(1,0){0.5}}

\linethickness{0.3mm}
\put(30,30){\line(0,1){10}}
\linethickness{0.3mm}
\put(40,30){\line(0,1){10}}
\put(40,25){\makebox(0,0)[cc]{$\theta_{3}$}}

\put(20,60){\makebox(0,0)[cc]{}}

\put(120,60){\makebox(0,0)[cc]{$a_{m}$}}

\put(20,60){\makebox(0,0)[cc]{$a_{1}=k-1$}}

\put(90,30){\makebox(0,0)[cc]{$\cdots$}}

\put(50,30){\line(0,1){10}}

\put(50,25){\makebox(0,0)[cc]{$\theta_{4}$}}

\put(50,60){\makebox(0,0)[cc]{$a_{4}=k+2$}}

\put(0,75){\makebox(0,0)[cc]{$\text{Case 1-a}$}}

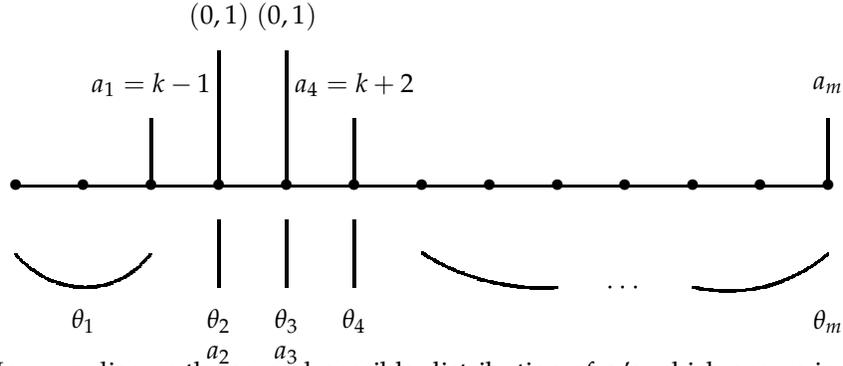
\put(30,20){\makebox(0,0)[cc]{$a_{2}$}}

\put(40,20){\makebox(0,0)[cc]{$a_{3}$}}

\put(120,25){\makebox(0,0)[cc]{$\theta_{m}$}}

\end{picture}

Now we discuss the second possible distribution of $a_{i}$'s which occurs in Case 1-b. In Case 1-b (diagram below) we set $a_{1}=k-1, a_{2}=k+1$ and $a_{3}$ can be chosen freely (similar to $a_{4}$ in Case 1-a) to have any value as long as $a_{3}\geq k+2$:

\ifx\JPicScale\undefined\def\JPicScale{.9}\fi
\unitlength \JPicScale mm
\begin{picture}(128.05,60)(-5,20)
\linethickness{0.3mm}
\put(0,45){\line(1,0){120}}
\put(0,45){\makebox(0,0)[cc]{$\bullet$}}

\put(10,45){\makebox(0,0)[cc]{$\bullet$}}

\put(20,45){\makebox(0,0)[cc]{$\bullet$}}

\put(30,45){\makebox(0,0)[cc]{$\bullet$}}

\put(40,45){\makebox(0,0)[cc]{$\bullet$}}

\put(50,45){\makebox(0,0)[cc]{$\bullet$}}

\put(60,45){\makebox(0,0)[cc]{$\bullet$}}

\put(70,45){\makebox(0,0)[cc]{$\bullet$}}

\put(80,45){\makebox(0,0)[cc]{$\bullet$}}

\put(90,45){\makebox(0,0)[cc]{$\bullet$}}

\put(100,45){\makebox(0,0)[cc]{$\bullet$}}

\put(110,45){\makebox(0,0)[cc]{$\bullet$}}

\put(120,45){\makebox(0,0)[cc]{$\bullet$}}

\linethickness{0.3mm}
\put(20,45){\line(0,1){10}}
\linethickness{0.3mm}
\put(50,45){\line(0,1){10}}
\linethickness{0.3mm}
\put(120,45){\line(0,1){10}}
\put(10,25){\makebox(0,0)[cc]{$\theta_{1}$}}

\linethickness{0.3mm}
\put(40,45){\line(0,1){20}}
\put(40,40){\makebox(0,0)[cc]{}}

\linethickness{0.3mm}
\put(30,45){\line(0,1){20}}
\put(30,70){\makebox(0,0)[cc]{$(0,1)$}}

\put(40,70){\makebox(0,0)[cc]{$(0,1)$}}

\linethickness{0.3mm}
\multiput(0,35)(0.1,-0.13){3}{\line(0,-1){0.13}}
\multiput(0.31,34.6)(0.11,-0.13){3}{\line(0,-1){0.13}}
\multiput(0.64,34.22)(0.11,-0.12){3}{\line(0,-1){0.12}}
\multiput(0.98,33.85)(0.12,-0.12){3}{\line(0,-1){0.12}}
\multiput(1.33,33.49)(0.12,-0.11){3}{\line(1,0){0.12}}
\multiput(1.7,33.15)(0.13,-0.11){3}{\line(1,0){0.13}}
\multiput(2.09,32.82)(0.13,-0.1){3}{\line(1,0){0.13}}
\multiput(2.48,32.51)(0.2,-0.15){2}{\line(1,0){0.2}}
\multiput(2.89,32.22)(0.21,-0.14){2}{\line(1,0){0.21}}
\multiput(3.31,31.94)(0.22,-0.13){2}{\line(1,0){0.22}}
\multiput(3.74,31.68)(0.22,-0.12){2}{\line(1,0){0.22}}
\multiput(4.19,31.43)(0.23,-0.11){2}{\line(1,0){0.23}}
\multiput(4.64,31.21)(0.23,-0.1){2}{\line(1,0){0.23}}
\multiput(5.1,31)(0.23,-0.09){2}{\line(1,0){0.23}}
\multiput(5.56,30.81)(0.47,-0.17){1}{\line(1,0){0.47}}
\multiput(6.04,30.64)(0.48,-0.15){1}{\line(1,0){0.48}}
\multiput(6.52,30.49)(0.49,-0.13){1}{\line(1,0){0.49}}
\multiput(7.01,30.36)(0.49,-0.11){1}{\line(1,0){0.49}}
\multiput(7.5,30.25)(0.5,-0.09){1}{\line(1,0){0.5}}
\multiput(7.99,30.16)(0.5,-0.07){1}{\line(1,0){0.5}}
\multiput(8.49,30.09)(0.5,-0.05){1}{\line(1,0){0.5}}
\multiput(8.99,30.04)(0.5,-0.03){1}{\line(1,0){0.5}}
\multiput(9.5,30.01)(0.5,-0.01){1}{\line(1,0){0.5}}
\multiput(10,30)(0.5,0.01){1}{\line(1,0){0.5}}
\multiput(10.5,30.01)(0.5,0.03){1}{\line(1,0){0.5}}
\multiput(11.01,30.04)(0.5,0.05){1}{\line(1,0){0.5}}
\multiput(11.51,30.09)(0.5,0.07){1}{\line(1,0){0.5}}
\multiput(12.01,30.16)(0.5,0.09){1}{\line(1,0){0.5}}
\multiput(12.5,30.25)(0.49,0.11){1}{\line(1,0){0.49}}
\multiput(12.99,30.36)(0.49,0.13){1}{\line(1,0){0.49}}
\multiput(13.48,30.49)(0.48,0.15){1}{\line(1,0){0.48}}
\multiput(13.96,30.64)(0.47,0.17){1}{\line(1,0){0.47}}
\multiput(14.44,30.81)(0.23,0.09){2}{\line(1,0){0.23}}
\multiput(14.9,31)(0.23,0.1){2}{\line(1,0){0.23}}
\multiput(15.36,31.21)(0.23,0.11){2}{\line(1,0){0.23}}
\multiput(15.81,31.43)(0.22,0.12){2}{\line(1,0){0.22}}
\multiput(16.26,31.68)(0.22,0.13){2}{\line(1,0){0.22}}
\multiput(16.69,31.94)(0.21,0.14){2}{\line(1,0){0.21}}
\multiput(17.11,32.22)(0.2,0.15){2}{\line(1,0){0.2}}
\multiput(17.52,32.51)(0.13,0.1){3}{\line(1,0){0.13}}
\multiput(17.91,32.82)(0.13,0.11){3}{\line(1,0){0.13}}
\multiput(18.3,33.15)(0.12,0.11){3}{\line(1,0){0.12}}
\multiput(18.67,33.49)(0.12,0.12){3}{\line(0,1){0.12}}
\multiput(19.02,33.85)(0.11,0.12){3}{\line(0,1){0.12}}
\multiput(19.36,34.22)(0.11,0.13){3}{\line(0,1){0.13}}
\multiput(19.69,34.6)(0.1,0.13){3}{\line(0,1){0.13}}

\linethickness{0.3mm}
\multiput(100,30)(0.49,-0.1){1}{\line(1,0){0.49}}
\multiput(100.49,29.9)(0.49,-0.09){1}{\line(1,0){0.49}}
\multiput(100.97,29.81)(0.49,-0.08){1}{\line(1,0){0.49}}
\multiput(101.47,29.72)(0.49,-0.07){1}{\line(1,0){0.49}}
\multiput(101.96,29.65)(0.49,-0.06){1}{\line(1,0){0.49}}
\multiput(102.45,29.59)(0.49,-0.05){1}{\line(1,0){0.49}}
\multiput(102.94,29.54)(0.5,-0.04){1}{\line(1,0){0.5}}
\multiput(103.44,29.5)(0.5,-0.03){1}{\line(1,0){0.5}}
\multiput(103.94,29.48)(0.5,-0.02){1}{\line(1,0){0.5}}
\multiput(104.43,29.46)(0.5,-0.01){1}{\line(1,0){0.5}}
\multiput(104.93,29.45)(0.5,0){1}{\line(1,0){0.5}}
\multiput(105.43,29.46)(0.5,0.01){1}{\line(1,0){0.5}}
\multiput(105.92,29.47)(0.5,0.03){1}{\line(1,0){0.5}}
\multiput(106.42,29.49)(0.5,0.04){1}{\line(1,0){0.5}}
\multiput(106.92,29.53)(0.49,0.05){1}{\line(1,0){0.49}}
\multiput(107.41,29.58)(0.49,0.06){1}{\line(1,0){0.49}}
\multiput(107.9,29.63)(0.49,0.07){1}{\line(1,0){0.49}}
\multiput(108.4,29.7)(0.49,0.08){1}{\line(1,0){0.49}}
\multiput(108.89,29.78)(0.49,0.09){1}{\line(1,0){0.49}}
\multiput(109.38,29.87)(0.49,0.1){1}{\line(1,0){0.49}}
\multiput(109.86,29.97)(0.48,0.11){1}{\line(1,0){0.48}}
\multiput(110.35,30.08)(0.48,0.12){1}{\line(1,0){0.48}}
\multiput(110.83,30.2)(0.48,0.13){1}{\line(1,0){0.48}}
\multiput(111.31,30.33)(0.48,0.14){1}{\line(1,0){0.48}}
\multiput(111.79,30.47)(0.47,0.15){1}{\line(1,0){0.47}}
\multiput(112.26,30.62)(0.47,0.16){1}{\line(1,0){0.47}}
\multiput(112.73,30.79)(0.47,0.17){1}{\line(1,0){0.47}}
\multiput(113.2,30.96)(0.23,0.09){2}{\line(1,0){0.23}}
\multiput(113.66,31.14)(0.23,0.1){2}{\line(1,0){0.23}}
\multiput(114.12,31.33)(0.23,0.1){2}{\line(1,0){0.23}}
\multiput(114.57,31.53)(0.22,0.11){2}{\line(1,0){0.22}}
\multiput(115.02,31.74)(0.22,0.11){2}{\line(1,0){0.22}}
\multiput(115.47,31.97)(0.22,0.12){2}{\line(1,0){0.22}}
\multiput(115.91,32.2)(0.22,0.12){2}{\line(1,0){0.22}}
\multiput(116.34,32.44)(0.22,0.12){2}{\line(1,0){0.22}}
\multiput(116.77,32.68)(0.21,0.13){2}{\line(1,0){0.21}}
\multiput(117.2,32.94)(0.21,0.13){2}{\line(1,0){0.21}}
\multiput(117.62,33.21)(0.21,0.14){2}{\line(1,0){0.21}}
\multiput(118.03,33.49)(0.2,0.14){2}{\line(1,0){0.2}}
\multiput(118.44,33.77)(0.2,0.15){2}{\line(1,0){0.2}}
\multiput(118.84,34.07)(0.13,0.1){3}{\line(1,0){0.13}}
\multiput(119.23,34.37)(0.13,0.1){3}{\line(1,0){0.13}}
\multiput(119.62,34.68)(0.13,0.11){3}{\line(1,0){0.13}}

\linethickness{0.3mm}
\multiput(60,35)(0.21,-0.13){2}{\line(1,0){0.21}}
\multiput(60.42,34.74)(0.21,-0.13){2}{\line(1,0){0.21}}
\multiput(60.85,34.48)(0.22,-0.13){2}{\line(1,0){0.22}}
\multiput(61.28,34.22)(0.22,-0.12){2}{\line(1,0){0.22}}
\multiput(61.72,33.98)(0.22,-0.12){2}{\line(1,0){0.22}}
\multiput(62.16,33.74)(0.22,-0.12){2}{\line(1,0){0.22}}
\multiput(62.6,33.51)(0.22,-0.11){2}{\line(1,0){0.22}}
\multiput(63.04,33.28)(0.22,-0.11){2}{\line(1,0){0.22}}
\multiput(63.49,33.07)(0.23,-0.11){2}{\line(1,0){0.23}}
\multiput(63.95,32.86)(0.23,-0.1){2}{\line(1,0){0.23}}
\multiput(64.4,32.65)(0.23,-0.1){2}{\line(1,0){0.23}}
\multiput(64.86,32.45)(0.23,-0.1){2}{\line(1,0){0.23}}
\multiput(65.32,32.26)(0.23,-0.09){2}{\line(1,0){0.23}}
\multiput(65.79,32.08)(0.47,-0.18){1}{\line(1,0){0.47}}
\multiput(66.25,31.91)(0.47,-0.17){1}{\line(1,0){0.47}}
\multiput(66.72,31.74)(0.47,-0.16){1}{\line(1,0){0.47}}
\multiput(67.2,31.58)(0.48,-0.15){1}{\line(1,0){0.48}}
\multiput(67.67,31.42)(0.48,-0.15){1}{\line(1,0){0.48}}
\multiput(68.15,31.27)(0.48,-0.14){1}{\line(1,0){0.48}}
\multiput(68.63,31.13)(0.48,-0.13){1}{\line(1,0){0.48}}
\multiput(69.11,31)(0.48,-0.12){1}{\line(1,0){0.48}}
\multiput(69.59,30.88)(0.49,-0.12){1}{\line(1,0){0.49}}
\multiput(70.08,30.76)(0.49,-0.11){1}{\line(1,0){0.49}}
\multiput(70.57,30.65)(0.49,-0.1){1}{\line(1,0){0.49}}
\multiput(71.06,30.55)(0.49,-0.09){1}{\line(1,0){0.49}}
\multiput(71.55,30.45)(0.49,-0.09){1}{\line(1,0){0.49}}
\multiput(72.04,30.36)(0.49,-0.08){1}{\line(1,0){0.49}}
\multiput(72.53,30.29)(0.49,-0.07){1}{\line(1,0){0.49}}
\multiput(73.02,30.21)(0.5,-0.06){1}{\line(1,0){0.5}}
\multiput(73.52,30.15)(0.5,-0.06){1}{\line(1,0){0.5}}
\multiput(74.02,30.09)(0.5,-0.05){1}{\line(1,0){0.5}}
\multiput(74.51,30.04)(0.5,-0.04){1}{\line(1,0){0.5}}
\multiput(75.01,30)(0.5,-0.03){1}{\line(1,0){0.5}}
\multiput(75.51,29.96)(0.5,-0.03){1}{\line(1,0){0.5}}
\multiput(76.01,29.94)(0.5,-0.02){1}{\line(1,0){0.5}}
\multiput(76.51,29.92)(0.5,-0.01){1}{\line(1,0){0.5}}
\multiput(77.01,29.91)(0.5,-0){1}{\line(1,0){0.5}}
\multiput(77.51,29.9)(0.5,0){1}{\line(1,0){0.5}}
\multiput(78,29.91)(0.5,0.01){1}{\line(1,0){0.5}}
\multiput(78.5,29.92)(0.5,0.02){1}{\line(1,0){0.5}}
\multiput(79,29.94)(0.5,0.03){1}{\line(1,0){0.5}}
\multiput(79.5,29.97)(0.5,0.03){1}{\line(1,0){0.5}}

\put(20,60){\makebox(0,0)[cc]{}}

\put(120,60){\makebox(0,0)[cc]{$a_{m}$}}

\put(20,60){\makebox(0,0)[cc]{$a_{1}=k-1$}}

\put(90,30){\makebox(0,0)[cc]{$\cdots$}}

\put(50,30){\line(0,1){10}}

\put(50,25){\makebox(0,0)[cc]{$\theta_{3}$}}

\put(0,75){\makebox(0,0)[cc]{$\text{Case 1-b}$}}

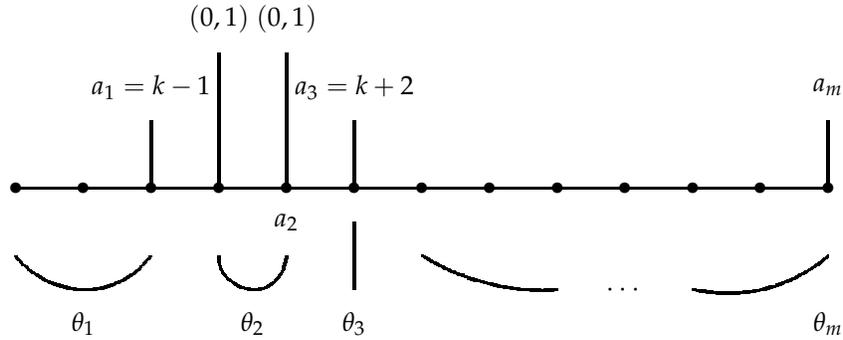
\put(50,60){\makebox(0,0)[cc]{$a_{3}=k+2$}}

\put(120,25){\makebox(0,0)[cc]{$\theta_{m}$}}

\put(40,40){\makebox(0,0)[cc]{$a_{2}$}}

\linethickness{0.3mm}
\multiput(30,35)(0.03,-0.51){1}{\line(0,-1){0.51}}
\multiput(30.03,34.49)(0.08,-0.5){1}{\line(0,-1){0.5}}
\multiput(30.1,33.99)(0.13,-0.49){1}{\line(0,-1){0.49}}
\multiput(30.23,33.5)(0.18,-0.47){1}{\line(0,-1){0.47}}
\multiput(30.41,33.03)(0.11,-0.23){2}{\line(0,-1){0.23}}
\multiput(30.63,32.57)(0.13,-0.21){2}{\line(0,-1){0.21}}
\multiput(30.9,32.14)(0.1,-0.13){3}{\line(0,-1){0.13}}
\multiput(31.21,31.74)(0.12,-0.12){3}{\line(0,-1){0.12}}
\multiput(31.56,31.38)(0.13,-0.11){3}{\line(1,0){0.13}}
\multiput(31.94,31.05)(0.21,-0.14){2}{\line(1,0){0.21}}
\multiput(32.36,30.76)(0.22,-0.12){2}{\line(1,0){0.22}}
\multiput(32.8,30.51)(0.23,-0.1){2}{\line(1,0){0.23}}
\multiput(33.26,30.31)(0.48,-0.15){1}{\line(1,0){0.48}}
\multiput(33.75,30.16)(0.5,-0.1){1}{\line(1,0){0.5}}
\multiput(34.24,30.06)(0.5,-0.05){1}{\line(1,0){0.5}}
\put(34.75,30.01){\line(1,0){0.51}}
\multiput(35.25,30.01)(0.5,0.05){1}{\line(1,0){0.5}}
\multiput(35.76,30.06)(0.5,0.1){1}{\line(1,0){0.5}}
\multiput(36.25,30.16)(0.48,0.15){1}{\line(1,0){0.48}}
\multiput(36.74,30.31)(0.23,0.1){2}{\line(1,0){0.23}}
\multiput(37.2,30.51)(0.22,0.12){2}{\line(1,0){0.22}}
\multiput(37.64,30.76)(0.21,0.14){2}{\line(1,0){0.21}}
\multiput(38.06,31.05)(0.13,0.11){3}{\line(1,0){0.13}}
\multiput(38.44,31.38)(0.12,0.12){3}{\line(0,1){0.12}}
\multiput(38.79,31.74)(0.1,0.13){3}{\line(0,1){0.13}}
\multiput(39.1,32.14)(0.13,0.21){2}{\line(0,1){0.21}}
\multiput(39.37,32.57)(0.11,0.23){2}{\line(0,1){0.23}}
\multiput(39.59,33.03)(0.18,0.47){1}{\line(0,1){0.47}}
\multiput(39.77,33.5)(0.13,0.49){1}{\line(0,1){0.49}}
\multiput(39.9,33.99)(0.08,0.5){1}{\line(0,1){0.5}}
\multiput(39.97,34.49)(0.03,0.51){1}{\line(0,1){0.51}}

\put(35,25){\makebox(0,0)[cc]{$\theta_{2}$}}

\end{picture}

Let us now compute the value of  $U^{1}_{k,t}$. Since Case 1 was given by two possible configurations (Case 1-a and Case 1-b) we denote by $U^{a}_{k,k+1}$ the value of $U^{1}_{k,t}$ when we have the Case 1-a configuration, and similarly by $U^{b}_{k,k+1}$ the value of $U^{1}_{k,t}$ when we have the Case 1-b configuration. It is trivially seen that $U^{1}_{k,t}=U^{a}_{k,k+1}+U^{b}_{k,k+1}$. We compute the coefficient $S^{1a}_{\textbf{E}_{2}}$ and $S^{1b}_{\textbf{E}_{2}}$ induced by the fixed distributions of $a_{i}$'s as shown in Case 1-a and Case 1-b. 

Consider the diagram of Case 1-a; We set for the variable $l$ in \eqref{eqU2}, $l=1$ or $l=2$, (for $l>2$, $S^{1a}_{\textbf{E}_{2}}=0$). If $l=1$ then according to formula  \eqref{eqU2} we need to compute $S^{1a}_{\textbf{E}_{2}}(\theta_{1},\cdots,\theta_{m})$. Note that $\tau^{\bullet}(\theta_{2})=-1$ and $\tau^{\bullet}(\theta_{3})=-1$ then $\tau^{\bullet}(\theta_{2})\leq \tau^{\bullet}(\theta_{3})$ however $\tilde{\tau}(\theta_{1}+\theta_{2})\ngtr \tilde{\tau}(\theta_{3}+\cdots+\theta_{m})$, hence neither condition $(a)$ nor $(b)$ in Definition \ref{SBp} are satisfied and $S^{1a}_{\textbf{E}_{2}}(\theta_{1},\cdots,\theta_{m})=0$. Now set $l=2$. Setting $l=2$ means that we need to choose $0=b_{0}<b_{1}<b_{2}=m$ so that  $b_{i}$, $i=0,1,2$, satisfy the conditions in Definition \eqref{eqU2}.  Note that one can choose $b_{1}=1,\cdots, m$. However the only allowed choice for $b_{1}$ is to set $b_{1}=2$. We explain this fact further; Set $b_{1}=1$, in that case $\gamma_{1}=\theta_{1}$ and $\gamma_{2}=\theta_{2}+\cdots +\theta_{m}$. This configuration is not allowed, since for $\gamma_{1}$, $\tilde{\tau}(\gamma_{1})=0\neq \tilde{\tau}(\beta,2)=1$. One easily observes that using similar arguments, the only allowable  choice is to set $b_{1}=2$. Now define:
\begin{align*}
& 
U^{a}_{k,k+1}=\sum_{\Lambda}\frac{-1}{2}S^{1a}_{\textbf{E}_{2}}(\theta_{1},\theta_{2})\cdot S^{1a}_{\textbf{E}_{2}}(\theta_{3},\cdots,\theta_{m})\cdot \prod^{m}_{i=1}\frac{1}{(a_{i}-a_{i-1})!},
\end{align*}
where by similar arguments $S^{1a}_{\textbf{E}_{2}}(\theta_{1},\theta_{2})=(-1)^{0}=1$ and $S^{1a}_{\textbf{E}_{2}}(\theta_{3},\cdots,\theta_{m})=(-1)^{(m-3)}$. Hence 
\begin{align*}
&
U^{a}_{k,k+1}=(-1)\cdot \sum_{\Lambda}\frac{1}{2}(-1)^{(m-3)}\cdot \prod^{m}_{i=1}\frac{1}{(a_{i}-a_{i-1})!}\notag\\
&
=(-1)\cdot\sum_{\Lambda}\frac{1}{2}(-1)^{(m-3)}\cdot \frac{1}{(a_{3}-a_{2})!}\cdot\frac{1}{(a_{2}-a_{1})!}\cdot \frac{1}{(a_{1}-a_{0})!}\cdot\prod^{m}_{i=4}\frac{1}{(a_{i}-a_{i-1})!}.
\end{align*}
By looking at configuration in Case 1-a, it is easy to see that $a_{0}=0, a_{1}=k-1$, $a_{2}=k$ and $a_{3}=k+1$. Hence $(a_{2}-a_{1})=1$ and $a_{1}-a_{0}=k-1$. Now we use the result of \cite[Lemma 13.9]{a30} and rewrite this equation as follows:
\begin{align}\label{case111}
&
U^{a}_{k,k+1}=(-\frac{1}{2})\cdot \frac{1}{(a_{3}-a_{2})!}\cdot \frac{1}{(a_{2}-a_{1})!}\cdot \frac{1}{(a_{1}-a_{0})!}\sum_{k+1\leq m\leq n}(-1)^{(m-3)}\cdot \prod^{m}_{i=4}\frac{1}{(a_{i}-a_{i-1})!}=\notag\\
&
(-\frac{1}{2})\cdot \frac{1}{(k-1)!}\cdot \frac{(-1)^{(n-(1+k))}}{(n-(1+k))!}.\notag\\
\end{align}
A similar analysis is carried out for the diagram in Case 1-b. Note that in this case $\theta_{2}=(0,1)+(0,1)=(0,2)$. We can set $l=1$ or $l=2$. Setting $l=2$ would result in obtaining a disallowed configuration, since there exists at least one $\gamma_{i}$ for $i=1,2$ so that $\tilde{\tau}(\gamma_{i})=0\neq \tilde{\tau}(\beta,2)=1$. Hence we set $l=1$. Define
\begin{align*}
U^{b}_{k,k+1}:=\sum_{\Lambda}S^{1b}_{\textbf{E}_{2}}(\theta_{1},\theta_{2},\theta_{3},\cdots,\theta_{m})\cdot \prod^{m}_{i=1}\frac{1}{(a_{i}-a_{i-1})!},
\end{align*}
where by similar arguments, $S^{1b}_{\textbf{E}_{2}}(\theta_{1},\cdots,\theta_{m})=(-1)^{(m-2)}$. Hence 
\begin{align*}
&
U^{b}_{k,k+1}=\sum_{\Lambda}(-1)^{(m-2)}\cdot \prod^{m}_{i=1}\frac{1}{(a_{i}-a_{i-1})!}\notag\\
&
=\sum_{\Lambda}(-1)^{(m-2)}\cdot \frac{1}{(a_{2}-a_{1})!}\cdot\frac{1}{(a_{1}-a_{0})!}\prod^{m}_{i=3}\frac{1}{(a_{i}-a_{i-1})!}.
\end{align*}
By the diagram in Case 1-b, it is easy to see that $a_{0}=0, a_{1}=k-1$ and $a_{2}=k+1$, hence $(a_{2}-a_{1})=2$ and $a_{1}-a_{0}=k-1$. Now use the result of \cite[Lemma 13.9]{a30} and rewrite this equation as follows:
\begin{align}\label{case11}
&
U^{b}_{k,k+1}=\frac{1}{(a_{1}-a_{0})!}\cdot \frac{1}{(a_{2}-a_{1})!}\cdot\sum_{k+1\leq m\leq n}(-1)^{(m-2)}\cdot \prod^{m}_{i=3}\frac{1}{(a_{i}-a_{i-1})!}\notag\\
&
=\frac{1}{2}\cdot \frac{1}{(k-1)!}\cdot \frac{(-1)^{(n-(1+k))}}{(n-(1+k))!}.
\end{align}
By definition $U^{1}_{k,t}=U^{a}_{k,k+1}+U^{b}_{k,k+1}$. Therefore, we obtain: 
\begin{align}\label{type12}
U^{1}_{k,t}=\frac{1}{2}\cdot \frac{1}{(k-1)!}\cdot \frac{(-1)^{(n-1-k)}}{(n-1-k)!}+(-\frac{1}{2})\cdot \frac{1}{(k-1)!}\cdot \frac{(-1)^{(n-1-k)}}{(n-1-k)!}=0.
\end{align}
\subsubsection{Computations in Case 2:}
Case 2 represents the configurations where there exists some $1\leq k \leq n$ for which (using our reparametrization convention) there exists  only one element of type $(\psi_{k},0)=(\beta_{k+1},0)$ between the two elements of type $(0,1)$ such that $\psi_k=\beta_{k+1}\neq 0$. Here we have that $U^{2}_{k,t}:=U_{k,k+2}$ (since there exists one $(\beta_{k+1},0)$ element in between the two $(0,1)$ elements, one in location $k$ and the second in location $t=k+2$). Moreover we denote the $S$ functions in this case by $S^{2}_{\textbf{E}_{2}}$. The set of allowable distributions for $a_{i}$'s is given as:

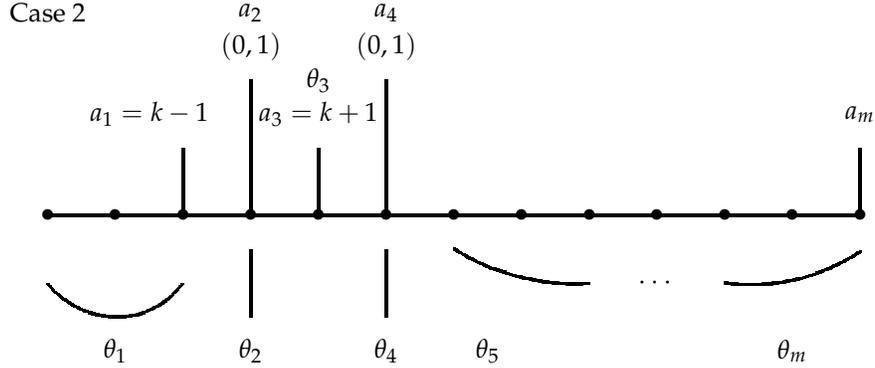
\ifx\JPicScale\undefined\def\JPicScale{.9}\fi
\unitlength \JPicScale mm
\begin{picture}(133.34,65)(-5,20)
\linethickness{0.3mm}
\put(0,45){\line(1,0){120}}
\put(0,45){\makebox(0,0)[cc]{$\bullet$}}

\put(10,45){\makebox(0,0)[cc]{$\bullet$}}

\put(20,45){\makebox(0,0)[cc]{$\bullet$}}

\put(30,45){\makebox(0,0)[cc]{$\bullet$}}

\put(40,45){\makebox(0,0)[cc]{$\bullet$}}

\put(50,45){\makebox(0,0)[cc]{$\bullet$}}

\put(60,45){\makebox(0,0)[cc]{$\bullet$}}

\put(70,45){\makebox(0,0)[cc]{$\bullet$}}

\put(80,45){\makebox(0,0)[cc]{$\bullet$}}

\put(90,45){\makebox(0,0)[cc]{$\bullet$}}

\put(100,45){\makebox(0,0)[cc]{$\bullet$}}

\put(110,45){\makebox(0,0)[cc]{$\bullet$}}

\put(120,45){\makebox(0,0)[cc]{$\bullet$}}

\put(0,75){\makebox(0,0)[cc]{$\text{Case 2}$}}

\linethickness{0.3mm}
\put(20,45){\line(0,1){10}}
\linethickness{0.3mm}
\put(120,45){\line(0,1){10}}
\put(10,25){\makebox(0,0)[cc]{$\theta_{1}$}}

\linethickness{0.3mm}
\put(50,45){\line(0,1){20}}
\linethickness{0.3mm}
\put(30,45){\line(0,1){20}}
\put(30,70){\makebox(0,0)[cc]{$(0,1)$}}

\put(50,70){\makebox(0,0)[cc]{$(0,1)$}}

\linethickness{0.3mm}
\put(40,45){\line(0,1){10}}
\linethickness{0.3mm}
\multiput(0,35)(0.1,-0.13){3}{\line(0,-1){0.13}}
\multiput(0.31,34.6)(0.11,-0.13){3}{\line(0,-1){0.13}}
\multiput(0.64,34.22)(0.11,-0.12){3}{\line(0,-1){0.12}}
\multiput(0.98,33.85)(0.12,-0.12){3}{\line(0,-1){0.12}}
\multiput(1.33,33.49)(0.12,-0.11){3}{\line(1,0){0.12}}
\multiput(1.7,33.15)(0.13,-0.11){3}{\line(1,0){0.13}}
\multiput(2.09,32.82)(0.13,-0.1){3}{\line(1,0){0.13}}
\multiput(2.48,32.51)(0.2,-0.15){2}{\line(1,0){0.2}}
\multiput(2.89,32.22)(0.21,-0.14){2}{\line(1,0){0.21}}
\multiput(3.31,31.94)(0.22,-0.13){2}{\line(1,0){0.22}}
\multiput(3.74,31.68)(0.22,-0.12){2}{\line(1,0){0.22}}
\multiput(4.19,31.43)(0.23,-0.11){2}{\line(1,0){0.23}}
\multiput(4.64,31.21)(0.23,-0.1){2}{\line(1,0){0.23}}
\multiput(5.1,31)(0.23,-0.09){2}{\line(1,0){0.23}}
\multiput(5.56,30.81)(0.47,-0.17){1}{\line(1,0){0.47}}
\multiput(6.04,30.64)(0.48,-0.15){1}{\line(1,0){0.48}}
\multiput(6.52,30.49)(0.49,-0.13){1}{\line(1,0){0.49}}
\multiput(7.01,30.36)(0.49,-0.11){1}{\line(1,0){0.49}}
\multiput(7.5,30.25)(0.5,-0.09){1}{\line(1,0){0.5}}
\multiput(7.99,30.16)(0.5,-0.07){1}{\line(1,0){0.5}}
\multiput(8.49,30.09)(0.5,-0.05){1}{\line(1,0){0.5}}
\multiput(8.99,30.04)(0.5,-0.03){1}{\line(1,0){0.5}}
\multiput(9.5,30.01)(0.5,-0.01){1}{\line(1,0){0.5}}
\multiput(10,30)(0.5,0.01){1}{\line(1,0){0.5}}
\multiput(10.5,30.01)(0.5,0.03){1}{\line(1,0){0.5}}
\multiput(11.01,30.04)(0.5,0.05){1}{\line(1,0){0.5}}
\multiput(11.51,30.09)(0.5,0.07){1}{\line(1,0){0.5}}
\multiput(12.01,30.16)(0.5,0.09){1}{\line(1,0){0.5}}
\multiput(12.5,30.25)(0.49,0.11){1}{\line(1,0){0.49}}
\multiput(12.99,30.36)(0.49,0.13){1}{\line(1,0){0.49}}
\multiput(13.48,30.49)(0.48,0.15){1}{\line(1,0){0.48}}
\multiput(13.96,30.64)(0.47,0.17){1}{\line(1,0){0.47}}
\multiput(14.44,30.81)(0.23,0.09){2}{\line(1,0){0.23}}
\multiput(14.9,31)(0.23,0.1){2}{\line(1,0){0.23}}
\multiput(15.36,31.21)(0.23,0.11){2}{\line(1,0){0.23}}
\multiput(15.81,31.43)(0.22,0.12){2}{\line(1,0){0.22}}
\multiput(16.26,31.68)(0.22,0.13){2}{\line(1,0){0.22}}
\multiput(16.69,31.94)(0.21,0.14){2}{\line(1,0){0.21}}
\multiput(17.11,32.22)(0.2,0.15){2}{\line(1,0){0.2}}
\multiput(17.52,32.51)(0.13,0.1){3}{\line(1,0){0.13}}
\multiput(17.91,32.82)(0.13,0.11){3}{\line(1,0){0.13}}
\multiput(18.3,33.15)(0.12,0.11){3}{\line(1,0){0.12}}
\multiput(18.67,33.49)(0.12,0.12){3}{\line(0,1){0.12}}
\multiput(19.02,33.85)(0.11,0.12){3}{\line(0,1){0.12}}
\multiput(19.36,34.22)(0.11,0.13){3}{\line(0,1){0.13}}
\multiput(19.69,34.6)(0.1,0.13){3}{\line(0,1){0.13}}

\put(15,60){\makebox(0,0)[cc]{$a_{1}=k-1$}}

\put(120,60){\makebox(0,0)[cc]{$a_{m}$}}

\linethickness{0.3mm}
\put(30,30){\line(0,1){10}}
\linethickness{0.3mm}
\put(50,30){\line(0,1){10}}
\linethickness{0.3mm}
\multiput(60,40)(0.21,-0.13){2}{\line(1,0){0.21}}
\multiput(60.42,39.74)(0.21,-0.13){2}{\line(1,0){0.21}}
\multiput(60.85,39.48)(0.22,-0.13){2}{\line(1,0){0.22}}
\multiput(61.28,39.22)(0.22,-0.12){2}{\line(1,0){0.22}}
\multiput(61.72,38.98)(0.22,-0.12){2}{\line(1,0){0.22}}
\multiput(62.16,38.74)(0.22,-0.12){2}{\line(1,0){0.22}}
\multiput(62.6,38.51)(0.22,-0.11){2}{\line(1,0){0.22}}
\multiput(63.04,38.28)(0.22,-0.11){2}{\line(1,0){0.22}}
\multiput(63.49,38.07)(0.23,-0.11){2}{\line(1,0){0.23}}
\multiput(63.95,37.86)(0.23,-0.1){2}{\line(1,0){0.23}}
\multiput(64.4,37.65)(0.23,-0.1){2}{\line(1,0){0.23}}
\multiput(64.86,37.45)(0.23,-0.1){2}{\line(1,0){0.23}}
\multiput(65.32,37.26)(0.23,-0.09){2}{\line(1,0){0.23}}
\multiput(65.79,37.08)(0.47,-0.18){1}{\line(1,0){0.47}}
\multiput(66.25,36.91)(0.47,-0.17){1}{\line(1,0){0.47}}
\multiput(66.72,36.74)(0.47,-0.16){1}{\line(1,0){0.47}}
\multiput(67.2,36.58)(0.48,-0.15){1}{\line(1,0){0.48}}
\multiput(67.67,36.42)(0.48,-0.15){1}{\line(1,0){0.48}}
\multiput(68.15,36.27)(0.48,-0.14){1}{\line(1,0){0.48}}
\multiput(68.63,36.13)(0.48,-0.13){1}{\line(1,0){0.48}}
\multiput(69.11,36)(0.48,-0.12){1}{\line(1,0){0.48}}
\multiput(69.59,35.88)(0.49,-0.12){1}{\line(1,0){0.49}}
\multiput(70.08,35.76)(0.49,-0.11){1}{\line(1,0){0.49}}
\multiput(70.57,35.65)(0.49,-0.1){1}{\line(1,0){0.49}}
\multiput(71.06,35.55)(0.49,-0.09){1}{\line(1,0){0.49}}
\multiput(71.55,35.45)(0.49,-0.09){1}{\line(1,0){0.49}}
\multiput(72.04,35.36)(0.49,-0.08){1}{\line(1,0){0.49}}
\multiput(72.53,35.29)(0.49,-0.07){1}{\line(1,0){0.49}}
\multiput(73.02,35.21)(0.5,-0.06){1}{\line(1,0){0.5}}
\multiput(73.52,35.15)(0.5,-0.06){1}{\line(1,0){0.5}}
\multiput(74.02,35.09)(0.5,-0.05){1}{\line(1,0){0.5}}
\multiput(74.51,35.04)(0.5,-0.04){1}{\line(1,0){0.5}}
\multiput(75.01,35)(0.5,-0.03){1}{\line(1,0){0.5}}
\multiput(75.51,34.96)(0.5,-0.03){1}{\line(1,0){0.5}}
\multiput(76.01,34.94)(0.5,-0.02){1}{\line(1,0){0.5}}
\multiput(76.51,34.92)(0.5,-0.01){1}{\line(1,0){0.5}}
\multiput(77.01,34.91)(0.5,-0){1}{\line(1,0){0.5}}
\multiput(77.51,34.9)(0.5,0){1}{\line(1,0){0.5}}
\multiput(78,34.91)(0.5,0.01){1}{\line(1,0){0.5}}
\multiput(78.5,34.92)(0.5,0.02){1}{\line(1,0){0.5}}
\multiput(79,34.94)(0.5,0.03){1}{\line(1,0){0.5}}
\multiput(79.5,34.97)(0.5,0.03){1}{\line(1,0){0.5}}

\linethickness{0.3mm}
\multiput(100,35)(0.5,-0.06){1}{\line(1,0){0.5}}
\multiput(100.5,34.94)(0.5,-0.05){1}{\line(1,0){0.5}}
\multiput(101,34.89)(0.5,-0.04){1}{\line(1,0){0.5}}
\multiput(101.51,34.85)(0.5,-0.03){1}{\line(1,0){0.5}}
\multiput(102.01,34.82)(0.5,-0.02){1}{\line(1,0){0.5}}
\multiput(102.51,34.8)(0.5,-0.01){1}{\line(1,0){0.5}}
\multiput(103.02,34.79)(0.5,-0.01){1}{\line(1,0){0.5}}
\multiput(103.52,34.78)(0.5,0){1}{\line(1,0){0.5}}
\multiput(104.03,34.78)(0.5,0.01){1}{\line(1,0){0.5}}
\multiput(104.53,34.79)(0.5,0.02){1}{\line(1,0){0.5}}
\multiput(105.04,34.81)(0.5,0.03){1}{\line(1,0){0.5}}
\multiput(105.54,34.84)(0.5,0.04){1}{\line(1,0){0.5}}
\multiput(106.04,34.88)(0.5,0.05){1}{\line(1,0){0.5}}
\multiput(106.55,34.92)(0.5,0.05){1}{\line(1,0){0.5}}
\multiput(107.05,34.98)(0.5,0.06){1}{\line(1,0){0.5}}
\multiput(107.55,35.04)(0.5,0.07){1}{\line(1,0){0.5}}
\multiput(108.05,35.11)(0.5,0.08){1}{\line(1,0){0.5}}
\multiput(108.55,35.19)(0.5,0.09){1}{\line(1,0){0.5}}
\multiput(109.04,35.28)(0.5,0.1){1}{\line(1,0){0.5}}
\multiput(109.54,35.37)(0.49,0.1){1}{\line(1,0){0.49}}
\multiput(110.03,35.48)(0.49,0.11){1}{\line(1,0){0.49}}
\multiput(110.52,35.59)(0.49,0.12){1}{\line(1,0){0.49}}
\multiput(111.01,35.71)(0.49,0.13){1}{\line(1,0){0.49}}
\multiput(111.5,35.84)(0.49,0.14){1}{\line(1,0){0.49}}
\multiput(111.99,35.98)(0.48,0.15){1}{\line(1,0){0.48}}
\multiput(112.47,36.13)(0.48,0.15){1}{\line(1,0){0.48}}
\multiput(112.95,36.28)(0.48,0.16){1}{\line(1,0){0.48}}
\multiput(113.43,36.44)(0.47,0.17){1}{\line(1,0){0.47}}
\multiput(113.9,36.61)(0.47,0.18){1}{\line(1,0){0.47}}
\multiput(114.37,36.79)(0.23,0.09){2}{\line(1,0){0.23}}
\multiput(114.84,36.98)(0.23,0.1){2}{\line(1,0){0.23}}
\multiput(115.31,37.17)(0.23,0.1){2}{\line(1,0){0.23}}
\multiput(115.77,37.38)(0.23,0.11){2}{\line(1,0){0.23}}
\multiput(116.23,37.59)(0.23,0.11){2}{\line(1,0){0.23}}
\multiput(116.69,37.81)(0.23,0.11){2}{\line(1,0){0.23}}
\multiput(117.14,38.03)(0.22,0.12){2}{\line(1,0){0.22}}
\multiput(117.58,38.26)(0.22,0.12){2}{\line(1,0){0.22}}
\multiput(118.03,38.5)(0.22,0.12){2}{\line(1,0){0.22}}
\multiput(118.47,38.75)(0.22,0.13){2}{\line(1,0){0.22}}
\multiput(118.9,39.01)(0.22,0.13){2}{\line(1,0){0.22}}
\multiput(119.33,39.27)(0.21,0.14){2}{\line(1,0){0.21}}
\multiput(119.76,39.54)(0.21,0.14){2}{\line(1,0){0.21}}

\put(90,35){\makebox(0,0)[cc]{$\cdots$}}

\put(30,25){\makebox(0,0)[cc]{$\theta_{2}$}}

\put(110,25){\makebox(0,0)[cc]{$\theta_{m}$}}

\put(40,60){\makebox(0,0)[cc]{$a_{3}=k+1$}}
\put(40,65){\makebox(0,0)[cc]{$\theta_{3}$}}

\put(30,75){\makebox(0,0)[cc]{$a_{2}$}}

\put(50,25){\makebox(0,0)[cc]{$\theta_{4}$}}

\put(65,25){\makebox(0,0)[cc]{$\theta_{5}$}}

\put(50,75){\makebox(0,0)[cc]{$a_{4}$}}

\end{picture}

Consider the diagram of Case 2; Here we can argue that the only possible value for $l$ in both diagrams is $l=2$. For $l=1$ consider $\theta_{2}$ and $\theta_{3}$ in the first diagram. Note that $\tau^{\bullet}(\theta_{2})\leq \tau^{\bullet}(\theta_{3})$ but $\tilde{\tau}(\theta_{1}+\theta_{2})\ngtr \tilde{\tau}(\theta_{3}+\cdots+ \theta_{m})$ hence $S^{2}_{\textbf{E}_{2}}(\theta_{1},\cdots,\theta_{m})=0$.  Setting $l=2$ means that we need to choose $0=b_{0}<b_{1}<b_{2}=m$ so that  $b_{i}$, $i=0,1,2$, satisfy the conditions in Definition \eqref{eqU2}.  Note that one can choose $b_{1}=2$ or $b_{1}=3$. We denote these values by choice (a) and (b) respectively. Set $b_{1}=2$ and define:
\begin{align*}
U^{a}_{k,k+2}:=\sum_{\Lambda}\frac{-1}{2}S^{2}_{\textbf{E}_{2}}(\theta_{1},\theta_{2})\cdot S^{2}_{\textbf{E}_{2}}(\theta_{3},\cdots,\theta_{m})\cdot \prod^{m}_{i=1}\frac{1}{(a_{i}-a_{i-1})!},
\end{align*}
Following similar computations, and via the result of \cite[Lemma 13.9]{a30}, we obtain the following identity:
\begin{align}\label{case21}
U^{a}_{k,k+2}=(-\frac{1}{2})\cdot\frac{1}{(k-1)!}\cdot \frac{(-1)^{(n-(k+2))}}{(n-(k+2))!}
\end{align}
Similarly, set $b_{1}=3$ and define:
\begin{align*}
U^{b}_{k,k+2}=\sum_{\Lambda}\frac{-1}{2}S^{2}_{\textbf{E}_{2}}(\theta_{1},\theta_{2},\theta_{3})\cdot S^{2}_{\textbf{E}_{2}}(\theta_{4},\cdots,\theta_{m})\cdot \prod^{m}_{i=1}\frac{1}{(a_{i}-a_{i-1})!},
\end{align*}
and obtain 
\begin{align*}
U^{b}_{k,k+2}=\sum_{\Lambda}\frac{1}{2}(-1)^{(m-4)}\cdot \prod^{m}_{i=1}\frac{1}{(a_{i}-a_{i-1})!}.
\end{align*}
By adding the contributions due to the two choices of $b_{1}=2$ and $b_{1}=3$, we obtain
\begin{align}\label{eshshak}
&
U^{2}_{k,t}=U^{a}_{k,k+2}+U^{b}_{k,k+2}=\notag\\
&
(-\frac{1}{2})\cdot\frac{1}{(k-1)!}\cdot  \frac{(-1)^{(n-(k+2))}}{(n-(k+2))!}+\frac{1}{2}\cdot\frac{1}{(k-1)!}\cdot \frac{(-1)^{(n-(k+2))}}{(n-(k+2))!}=0
\end{align}
\subsubsection{Computations in Case 3:} 
Case 3 represents the configurations where for some $1 \leq k<t\leq n$, there exists at least 2 elements of type $(\beta_{i},0)$ and $(\beta_{j},0)$  between the two elements of type $(0,1)$ such that $\beta_{i}\neq 0$ and $\beta_{j}\neq 0$.

\ifx\JPicScale\undefined\def\JPicScale{1}\fi
\unitlength \JPicScale mm
\begin{picture}(126.11,60)(-5,20)
\linethickness{0.3mm}
\put(0,45){\line(1,0){120}}
\put(0,45){\makebox(0,0)[cc]{$\bullet$}}

\put(10,45){\makebox(0,0)[cc]{$\bullet$}}

\put(20,45){\makebox(0,0)[cc]{$\bullet$}}

\put(30,45){\makebox(0,0)[cc]{$\bullet$}}

\put(40,45){\makebox(0,0)[cc]{$\bullet$}}

\put(50,45){\makebox(0,0)[cc]{$\bullet$}}

\put(60,45){\makebox(0,0)[cc]{$\bullet$}}

\put(70,45){\makebox(0,0)[cc]{$\bullet$}}

\put(80,45){\makebox(0,0)[cc]{$\bullet$}}

\put(90,45){\makebox(0,0)[cc]{$\bullet$}}

\put(100,45){\makebox(0,0)[cc]{$\bullet$}}

\put(110,45){\makebox(0,0)[cc]{$\bullet$}}

\put(120,45){\makebox(0,0)[cc]{$\bullet$}}

\put(0,75){\makebox(0,0)[cc]{$\text{Case 3}$}}

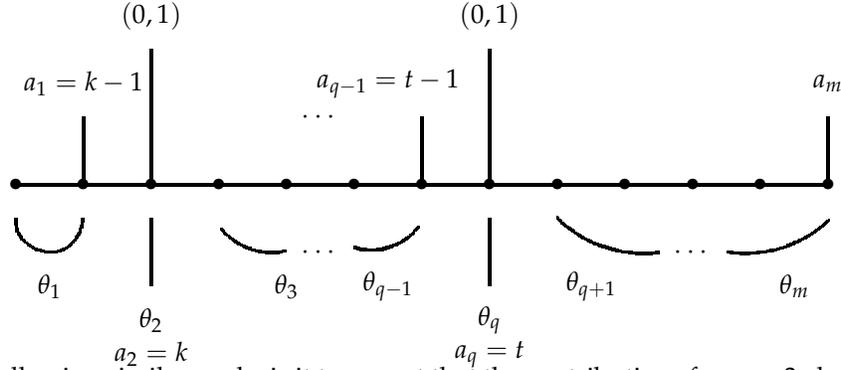
\linethickness{0.3mm}
\put(10,45){\line(0,1){10}}
\linethickness{0.3mm}
\put(120,45){\line(0,1){10}}
\put(5,30){\makebox(0,0)[cc]{$\theta_{1}$}}

\linethickness{0.3mm}
\put(70,45){\line(0,1){20}}
\linethickness{0.3mm}
\put(20,45){\line(0,1){20}}
\put(20,70){\makebox(0,0)[cc]{$(0,1)$}}

\put(70,70){\makebox(0,0)[cc]{$(0,1)$}}

\linethickness{0.3mm}
\put(60,45){\line(0,1){10}}
\linethickness{0.3mm}
\multiput(0,40)(0.02,-0.5){1}{\line(0,-1){0.5}}
\multiput(0.02,39.5)(0.07,-0.5){1}{\line(0,-1){0.5}}
\multiput(0.09,39)(0.12,-0.49){1}{\line(0,-1){0.49}}
\multiput(0.21,38.51)(0.17,-0.47){1}{\line(0,-1){0.47}}
\multiput(0.38,38.03)(0.11,-0.23){2}{\line(0,-1){0.23}}
\multiput(0.6,37.58)(0.13,-0.22){2}{\line(0,-1){0.22}}
\multiput(0.86,37.15)(0.1,-0.13){3}{\line(0,-1){0.13}}
\multiput(1.17,36.75)(0.12,-0.12){3}{\line(0,-1){0.12}}
\multiput(1.52,36.38)(0.13,-0.11){3}{\line(1,0){0.13}}
\multiput(1.9,36.05)(0.21,-0.15){2}{\line(1,0){0.21}}
\multiput(2.31,35.76)(0.22,-0.12){2}{\line(1,0){0.22}}
\multiput(2.75,35.51)(0.23,-0.1){2}{\line(1,0){0.23}}
\multiput(3.21,35.31)(0.48,-0.15){1}{\line(1,0){0.48}}
\multiput(3.69,35.16)(0.49,-0.1){1}{\line(1,0){0.49}}
\multiput(4.19,35.06)(0.5,-0.05){1}{\line(1,0){0.5}}
\put(4.69,35.01){\line(1,0){0.5}}
\multiput(5.19,35.01)(0.5,0.05){1}{\line(1,0){0.5}}
\multiput(5.69,35.06)(0.49,0.1){1}{\line(1,0){0.49}}
\multiput(6.19,35.16)(0.48,0.15){1}{\line(1,0){0.48}}
\multiput(6.67,35.31)(0.23,0.1){2}{\line(1,0){0.23}}
\multiput(7.13,35.51)(0.22,0.12){2}{\line(1,0){0.22}}
\multiput(7.57,35.76)(0.21,0.15){2}{\line(1,0){0.21}}
\multiput(7.98,36.05)(0.13,0.11){3}{\line(1,0){0.13}}
\multiput(8.36,36.38)(0.12,0.12){3}{\line(0,1){0.12}}
\multiput(8.71,36.75)(0.1,0.13){3}{\line(0,1){0.13}}
\multiput(9.02,37.15)(0.13,0.22){2}{\line(0,1){0.22}}
\multiput(9.28,37.58)(0.11,0.23){2}{\line(0,1){0.23}}
\multiput(9.5,38.03)(0.17,0.47){1}{\line(0,1){0.47}}
\multiput(9.67,38.51)(0.12,0.49){1}{\line(0,1){0.49}}
\multiput(9.79,39)(0.07,0.5){1}{\line(0,1){0.5}}
\multiput(9.86,39.5)(0.02,0.5){1}{\line(0,1){0.5}}

\put(10,60){\makebox(0,0)[cc]{$a_{1}=k-1$}}

\put(120,60){\makebox(0,0)[cc]{$a_{m}$}}

\linethickness{0.3mm}
\put(20,30){\line(0,1){10}}
\linethickness{0.3mm}
\multiput(79.9,40.1)(0.12,-0.12){3}{\line(1,0){0.12}}
\multiput(80.26,39.75)(0.12,-0.11){3}{\line(1,0){0.12}}
\multiput(80.62,39.41)(0.13,-0.11){3}{\line(1,0){0.13}}
\multiput(81,39.07)(0.13,-0.11){3}{\line(1,0){0.13}}
\multiput(81.38,38.75)(0.13,-0.1){3}{\line(1,0){0.13}}
\multiput(81.78,38.44)(0.2,-0.15){2}{\line(1,0){0.2}}
\multiput(82.18,38.15)(0.21,-0.14){2}{\line(1,0){0.21}}
\multiput(82.59,37.86)(0.21,-0.14){2}{\line(1,0){0.21}}
\multiput(83.01,37.59)(0.21,-0.13){2}{\line(1,0){0.21}}
\multiput(83.44,37.32)(0.22,-0.13){2}{\line(1,0){0.22}}
\multiput(83.87,37.07)(0.22,-0.12){2}{\line(1,0){0.22}}
\multiput(84.31,36.83)(0.22,-0.11){2}{\line(1,0){0.22}}
\multiput(84.76,36.61)(0.23,-0.11){2}{\line(1,0){0.23}}
\multiput(85.21,36.39)(0.23,-0.1){2}{\line(1,0){0.23}}
\multiput(85.67,36.19)(0.23,-0.09){2}{\line(1,0){0.23}}
\multiput(86.14,36.01)(0.47,-0.17){1}{\line(1,0){0.47}}
\multiput(86.61,35.83)(0.47,-0.16){1}{\line(1,0){0.47}}
\multiput(87.08,35.67)(0.48,-0.15){1}{\line(1,0){0.48}}
\multiput(87.56,35.53)(0.48,-0.13){1}{\line(1,0){0.48}}
\multiput(88.04,35.39)(0.49,-0.12){1}{\line(1,0){0.49}}
\multiput(88.53,35.27)(0.49,-0.11){1}{\line(1,0){0.49}}
\multiput(89.02,35.17)(0.49,-0.09){1}{\line(1,0){0.49}}
\multiput(89.51,35.08)(0.49,-0.08){1}{\line(1,0){0.49}}
\multiput(90.01,35)(0.5,-0.06){1}{\line(1,0){0.5}}
\multiput(90.5,34.94)(0.5,-0.05){1}{\line(1,0){0.5}}
\multiput(91,34.89)(0.5,-0.04){1}{\line(1,0){0.5}}
\multiput(91.5,34.85)(0.5,-0.02){1}{\line(1,0){0.5}}
\multiput(92,34.83)(0.5,-0.01){1}{\line(1,0){0.5}}
\multiput(92.5,34.82)(0.5,0.01){1}{\line(1,0){0.5}}
\multiput(93,34.83)(0.5,0.02){1}{\line(1,0){0.5}}
\multiput(93.5,34.85)(0.5,0.04){1}{\line(1,0){0.5}}
\multiput(94,34.89)(0.5,0.05){1}{\line(1,0){0.5}}
\multiput(94.5,34.94)(0.5,0.06){1}{\line(1,0){0.5}}

\linethickness{0.3mm}
\multiput(105,35)(0.5,-0.06){1}{\line(1,0){0.5}}
\multiput(105.5,34.94)(0.5,-0.05){1}{\line(1,0){0.5}}
\multiput(105.99,34.89)(0.5,-0.03){1}{\line(1,0){0.5}}
\multiput(106.49,34.86)(0.5,-0.02){1}{\line(1,0){0.5}}
\multiput(106.99,34.84)(0.5,-0.01){1}{\line(1,0){0.5}}
\multiput(107.49,34.83)(0.5,0.01){1}{\line(1,0){0.5}}
\multiput(107.98,34.84)(0.5,0.02){1}{\line(1,0){0.5}}
\multiput(108.48,34.86)(0.5,0.03){1}{\line(1,0){0.5}}
\multiput(108.98,34.89)(0.5,0.05){1}{\line(1,0){0.5}}
\multiput(109.48,34.94)(0.5,0.06){1}{\line(1,0){0.5}}
\multiput(109.97,35)(0.49,0.07){1}{\line(1,0){0.49}}
\multiput(110.47,35.07)(0.49,0.09){1}{\line(1,0){0.49}}
\multiput(110.96,35.16)(0.49,0.1){1}{\line(1,0){0.49}}
\multiput(111.45,35.25)(0.49,0.11){1}{\line(1,0){0.49}}
\multiput(111.93,35.37)(0.48,0.13){1}{\line(1,0){0.48}}
\multiput(112.41,35.49)(0.48,0.14){1}{\line(1,0){0.48}}
\multiput(112.89,35.63)(0.48,0.15){1}{\line(1,0){0.48}}
\multiput(113.37,35.78)(0.47,0.16){1}{\line(1,0){0.47}}
\multiput(113.84,35.94)(0.47,0.18){1}{\line(1,0){0.47}}
\multiput(114.31,36.12)(0.23,0.09){2}{\line(1,0){0.23}}
\multiput(114.77,36.31)(0.23,0.1){2}{\line(1,0){0.23}}
\multiput(115.22,36.51)(0.23,0.11){2}{\line(1,0){0.23}}
\multiput(115.68,36.72)(0.22,0.11){2}{\line(1,0){0.22}}
\multiput(116.12,36.95)(0.22,0.12){2}{\line(1,0){0.22}}
\multiput(116.56,37.19)(0.22,0.12){2}{\line(1,0){0.22}}
\multiput(116.99,37.43)(0.21,0.13){2}{\line(1,0){0.21}}
\multiput(117.42,37.69)(0.21,0.14){2}{\line(1,0){0.21}}
\multiput(117.84,37.97)(0.21,0.14){2}{\line(1,0){0.21}}
\multiput(118.25,38.25)(0.2,0.15){2}{\line(1,0){0.2}}
\multiput(118.65,38.54)(0.13,0.1){3}{\line(1,0){0.13}}
\multiput(119.05,38.85)(0.13,0.1){3}{\line(1,0){0.13}}
\multiput(119.43,39.16)(0.13,0.11){3}{\line(1,0){0.13}}
\multiput(119.81,39.49)(0.12,0.11){3}{\line(1,0){0.12}}

\put(20,25){\makebox(0,0)[cc]{$\theta_{2}$}}

\put(115,30){\makebox(0,0)[cc]{$\theta_{m}$}}

\put(100,35){\makebox(0,0)[cc]{$\cdots$}}

\put(45,55){\makebox(0,0)[cc]{$\cdots$}}

\put(55,60){\makebox(0,0)[cc]{$a_{q-1}=t-1$}}

\linethickness{0.3mm}
\multiput(30.38,38.56)(0.1,-0.13){3}{\line(0,-1){0.13}}
\multiput(30.69,38.16)(0.11,-0.13){3}{\line(0,-1){0.13}}
\multiput(31.03,37.77)(0.12,-0.12){3}{\line(0,-1){0.12}}
\multiput(31.38,37.41)(0.13,-0.11){3}{\line(1,0){0.13}}
\multiput(31.76,37.07)(0.13,-0.11){3}{\line(1,0){0.13}}
\multiput(32.16,36.75)(0.21,-0.15){2}{\line(1,0){0.21}}
\multiput(32.57,36.45)(0.21,-0.14){2}{\line(1,0){0.21}}
\multiput(33,36.18)(0.22,-0.12){2}{\line(1,0){0.22}}
\multiput(33.44,35.93)(0.23,-0.11){2}{\line(1,0){0.23}}
\multiput(33.9,35.71)(0.24,-0.1){2}{\line(1,0){0.24}}
\multiput(34.37,35.52)(0.48,-0.17){1}{\line(1,0){0.48}}
\multiput(34.85,35.35)(0.49,-0.14){1}{\line(1,0){0.49}}
\multiput(35.34,35.21)(0.5,-0.11){1}{\line(1,0){0.5}}
\multiput(35.84,35.1)(0.5,-0.08){1}{\line(1,0){0.5}}
\multiput(36.34,35.02)(0.51,-0.05){1}{\line(1,0){0.51}}
\multiput(36.85,34.97)(0.51,-0.02){1}{\line(1,0){0.51}}
\multiput(37.36,34.95)(0.51,0.01){1}{\line(1,0){0.51}}
\multiput(37.86,34.95)(0.51,0.04){1}{\line(1,0){0.51}}
\multiput(38.37,34.99)(0.5,0.06){1}{\line(1,0){0.5}}
\multiput(38.88,35.05)(0.5,0.09){1}{\line(1,0){0.5}}
\multiput(39.38,35.15)(0.49,0.12){1}{\line(1,0){0.49}}

\linethickness{0.3mm}
\multiput(49.86,35.66)(0.48,-0.12){1}{\line(1,0){0.48}}
\multiput(50.34,35.54)(0.48,-0.09){1}{\line(1,0){0.48}}
\multiput(50.82,35.45)(0.49,-0.06){1}{\line(1,0){0.49}}
\multiput(51.31,35.39)(0.49,-0.04){1}{\line(1,0){0.49}}
\multiput(51.81,35.35)(0.49,-0.01){1}{\line(1,0){0.49}}
\multiput(52.3,35.34)(0.49,0.01){1}{\line(1,0){0.49}}
\multiput(52.79,35.35)(0.49,0.04){1}{\line(1,0){0.49}}
\multiput(53.28,35.39)(0.49,0.07){1}{\line(1,0){0.49}}
\multiput(53.77,35.45)(0.48,0.09){1}{\line(1,0){0.48}}
\multiput(54.26,35.55)(0.48,0.12){1}{\line(1,0){0.48}}
\multiput(54.74,35.66)(0.47,0.14){1}{\line(1,0){0.47}}
\multiput(55.21,35.8)(0.46,0.17){1}{\line(1,0){0.46}}
\multiput(55.67,35.97)(0.23,0.1){2}{\line(1,0){0.23}}
\multiput(56.13,36.16)(0.22,0.11){2}{\line(1,0){0.22}}
\multiput(56.57,36.37)(0.22,0.12){2}{\line(1,0){0.22}}
\multiput(57,36.61)(0.21,0.13){2}{\line(1,0){0.21}}
\multiput(57.42,36.87)(0.2,0.14){2}{\line(1,0){0.2}}
\multiput(57.83,37.15)(0.13,0.1){3}{\line(1,0){0.13}}
\multiput(58.22,37.46)(0.12,0.11){3}{\line(1,0){0.12}}
\multiput(58.59,37.78)(0.12,0.11){3}{\line(1,0){0.12}}
\multiput(58.95,38.12)(0.11,0.12){3}{\line(0,1){0.12}}
\multiput(59.28,38.48)(0.11,0.13){3}{\line(0,1){0.13}}

\put(40,30){\makebox(0,0)[cc]{$\theta_{3}$}}

\put(55,30){\makebox(0,0)[cc]{$\theta_{q-1}$}}

\linethickness{0.3mm}
\put(70,30){\line(0,1){10}}
\put(70,25){\makebox(0,0)[cc]{$\theta_{q}$}}

\put(85,30){\makebox(0,0)[cc]{$\theta_{q+1}$}}

\put(45,35){\makebox(0,0)[cc]{$\cdots$}}

\put(20,20){\makebox(0,0)[cc]{$a_{2}=k$}}

\put(70,20){\makebox(0,0)[cc]{$a_{q}=t$}}

\end{picture}

Following similar analysis it turns out that the contributions for case 3 also add up to zero, i.e: 
\begin{align} \label{case331}
&
U^{3}_{k,t}=\sum_{0<a_{0}<\cdots<a_{m}}\frac{(-1)}{2}\cdot\Bigg[(-1)^{(q-3)}\cdot (-1)^{(m-q)}+(-1)^{(q-2)}\notag\\
&
\cdot (-1)^{(m-q)}\Bigg]\cdot \prod^{m}_{i=1}\frac{1}{(a_{i}-a_{i-1})!}=0.\notag\\
\end{align}
We conclude that the contributions in Cases 1, 2 and 3 are all equal to zero. Hence $$U_{k,t}=U^{1}_{k,t}+U^{2}_{k,t}+U^{3}_{k,t}=0.$$ This finishes the proof of Lemma \ref{UB}.
\end{proof}
Recall that for $\textbf{E}_{1}$ in Equation \eqref{A}, the $(n-1)$'th class, $\beta_{n-1}$ was placed in the $n$'th spot, hence by change of variable $n$ to $l-1$, the equation for $\textbf{E}_{1}$ is given as:
\begin{align}\label{coeff-A}  
&
\textbf{E}_{1}=\sum_{\begin{subarray}{1} \beta_{1},\cdots, \beta_{l}\in C(\mathcal{A}_{p})\\ \beta_{1}+\cdots+\beta_{l}=\beta \end{subarray}}\frac{(-1)^{l}}{(l)!}\cdot\bar{\epsilon}^{(0,2)}(\tau^{\bullet})*\cdots *\bar{\epsilon}^{(\beta_{l},0)}(\tau^{\bullet})
+\sum_{\begin{subarray}{1}1\leq k\leq l\\\beta_{1},\cdots, \beta_{l}\in C(\mathcal{A}_{p})\\ \beta_{1}+\cdots+\beta_{l}=\beta \end{subarray}}\frac{(-1)^{l-k}}{(k-1)!(l-k)!}\notag\\
&
\cdot\bar{\epsilon}^{(\beta_{1},0)}(\tau^{\bullet})*\cdots 
\cdots*\bar{\epsilon}^{(\beta_{k},0)}(\tau^{\bullet})*\bar{\epsilon}^{(0,2)}(\tau^{\bullet})*\bar{\epsilon}^{(\beta_{k+1},0)}(\tau^{\bullet})*\cdots*\bar{\epsilon}^{(\beta_{l},0)}(\tau^{\bullet})\notag\\
&
=\sum_{\begin{subarray}{1}0\leq k\leq l\\ \beta_{1},\cdots, \beta_{l}\in C(\mathcal{A}_{p})\\ \beta_{1}+\cdots+\beta_{l}=\beta \end{subarray}}\frac{(-1)^{l-k}}{k!(l-k)!}\cdot\bar{\epsilon}^{(\beta_{1},0)}(\tau^{\bullet})*\cdots *\bar{\epsilon}^{(\beta_{k},0)}(\tau^{\bullet})
*\bar{\epsilon}^{(0,2)}(\tau^{\bullet})*\bar{\epsilon}^{(\beta_{k+1},0)}(\tau^{\bullet})\notag\\
&
*\cdots*\bar{\epsilon}^{(\beta_{l},0)}(\tau^{\bullet})\notag\\
\end{align}
The coefficients in \eqref{coeff-A} are precisely equal to those appearing in  \cite[Equation (13.25)]{a30}. By rewriting the product of stack functions in terms of a nested brackets we obtain an equation analogous to the computation of Joyce and Song in \cite[Equation 13.26]{a30}. Simply replace $\bar{\epsilon}^{(0,1)}(\tau^{\bullet})$ in \cite[Equation (13.26)]{a30} with $\bar{\epsilon}^{(0,2)}(\tau^{\bullet})$ and obtain:
\begin{align}\label{epsilon2}
&
\bar{\epsilon}^{(\beta,2)}(\tilde{\tau})=\sum_{1\leq l,\beta_{1}+\cdots+\beta_{l}=\beta}\frac{(-1)^{l}}{l!}[[\cdots[[\bar{\epsilon}^{(0,2)}(\tau^{\bullet}),\bar{\epsilon}^{(\beta_{1},0)}(\tau^{\bullet})],\bar{\epsilon}^{(\beta_{2},0)}(\tau^{\bullet})],
\notag\\
&
\cdots],\bar{\epsilon}^{(\beta_{l},0)}(\tau^{\bullet})]\notag\\
\end{align}
\subsection{Wallcrossing for numerical invariants}\label{stable=semistable}
In this section we use Equation \eqref{epsilon2} to compute the wallcrossing identity between invariants of $\tilde{\tau}$-semistable objects in $\mathcal{B}_{p}$ and the generalized Donaldson--Thomas invariants. 
\begin{prop}\label{wall-cross}
(a). Let $\nu^{(\beta,0)}_{\mathfrak{M}_{\mathcal{B}_{p}}}$ and $\nu^{\beta}_{\mathfrak{M}}$ denote Behrend's constructible functions \cite[Section 1.3]{a1} on the moduli stack of objects in $\mathcal{B}_{p}$ (with fixed class $(\beta,0)$) and the moduli stack of sheaves with class $\beta$ respectively. The following identity holds true:
\begin{equation*}
\nu^{(\beta,0)}_{\mathfrak{M}_{\mathcal{B}_{p}}}\equiv \pi_{0}^{*}(\nu^{\beta}_{\mathfrak{M}})
\end{equation*}
where $\pi_{0}$ is the map $\pi_{0}:\mathfrak{M}^{(\beta,0)}_{\mathcal{B}_{p}}\rightarrow \mathfrak{M}^{\beta}$ which sends $(F,0,0)$ with $[(F,0,0)]=(\beta,0)$ to $F$ with Chern character $\beta$.\\
\end{prop}
\begin{proof} 
This is proven in \cite[Proposition 13.12]{a30}. 
\end{proof}
\begin{defn}
\cite[Definition 13.3]{a30}. Define the Euler form on $\mathcal{B}_{p}$ as $\bar{\chi}_{\mathcal{B}_{p}}:K(\mathcal{B}_{p})\times K(\mathcal{B}_{p})\to \mathbb{Z}$ such that
\begin{equation}
\bar{\chi}_{\mathcal{B}_{p}}((\beta,d), (\gamma,e))=\bar{\chi}(\beta,\gamma)-d\bar{\chi}([\mathcal{O}_{X}(-n)], \gamma)+e\bar{\chi}([\mathcal{O}_{X}(-n)],\beta),
\end{equation} 
where $\bar{\chi}()$ is the Euler form on $K(\text{coh}(X))$.
\end{defn}
\begin{defn}\label{13.11}\footnote{Look at \cite[Definition 13.11]{a30}. }
Define $\mathcal{S}$ to be the subset of $(\beta,d)$ in $C(\mathcal{B}_{p})\subset K(\mathcal{B}_{p})$ such that $P_{\beta}(t)=\frac{k}{d!}p(t)$ for $k=0,\cdots, N$ and $d=0$ or 1 or 2. Then $\mathcal{S}$ is a finite set \cite[Theorem 3.37]{a9}. Define a Lie algebra $\tilde{L}(\mathcal{B}_{p})$ to be the $\mathbb{Q}$-vector space with the basis of symbols $\tilde{\lambda}^{(\beta,d)}$ with $(\beta,d)\in \mathcal{S}$ with the Lie bracket
\begin{equation}\label{lambda}
[\tilde{\lambda}^{(\beta,d)},\tilde{\lambda}^{(\gamma,e)}]=(-1)^{\bar{\chi}_{\mathcal{B}_{p}}((\beta,d),(\gamma,e))}\bar{\chi}_{\mathcal{B}_{p}}((\beta,d),(\gamma,e))\tilde{\lambda}^{(\beta+\gamma,d+e)}
\end{equation}
for $(\beta+\gamma,d+e)\in \mathcal{S}$ and $[\tilde{\lambda}^{(\beta,d)},\tilde{\lambda}^{(\gamma,e)}]=0$ otherwise. Here It can be seen that $\bar{\chi}_{\mathcal{B}_{p}}$ is antisymmetric and hence, equation \eqref{lambda} satisfies the Jacobi-identity and that makes $\tilde{L}(\mathcal{B}_{p})$ into a finite-dimensional nilpotent Lie algebra over $\mathbb{Q}$. 
\end{defn}
In order to define the Lie algebra morphism $\tilde{\Psi}^{\mathcal{B}_{p}}:\textbf{SF}^{ind}_{al}\mathfrak{M}_{\mathcal{B}_{p}}\rightarrow \tilde{L}(\mathcal{B}_{p})$, apply \cite[Definition 5.13]{a30} to the moduli stack $\mathfrak{M}_{\mathcal{B}_{p}}$ and $\tilde{L}(\mathcal{B}_{p})$. Now we study the image of $\bar{\epsilon}^{(\beta,2)}(\tilde{\tau})$, $\bar{\epsilon}^{(0,2)}(\tau^{\bullet})$, $\bar{\epsilon}^{(\beta_{i},0)}(\tau^{\bullet})$ and $\bar{\epsilon}^{(0,1)}(\tau^{\bullet})$ under the morphism $\tilde{\Psi}^{\mathcal{B}_{p}}$:
\begin{defn}\label{Bp-ss-defn}
Define the invariant $\textbf{B}^{ss}_{p}(X,\beta,2,\tilde{\tau})$ associated to $\tilde{\tau}$-semistable objects of type $(\beta,2)$ in $\mathcal{B}_{p}$ by$$\tilde{\Psi}^{\mathcal{B}_{p}}(\bar{\epsilon}^{(\beta,2)}(\tilde{\tau}))=\textbf{B}^{ss}_{p}(X,\beta,2,\tilde{\tau})\cdot\tilde{\lambda}^{(\beta,2)},$$
where $\tilde{\Psi}^{\mathcal{B}_{p}}$ is given by the Lie algebra morphism defined in \cite[Section 13.4]{a30}.
\end{defn}
The next identity is proved by Joyce and Song in \cite[Section 13.5]{a30}:
\begin{equation}\label{0,1} 
\tilde{\Psi}^{\mathcal{B}_{p}}(\bar{\epsilon}^{(0,1)}(\tau^{\bullet}))=-\tilde{\lambda}^{(0,1)}.
\end{equation}
Now consider the decomposition $\beta=\displaystyle{\sum}_{i}\beta_{i}$ where $\beta_{i}$ is irreducible, then
\begin{equation}\label{1,0} 
\tilde{\Psi}^{\mathcal{B}_{p}}(\bar{\epsilon}^{(\beta_{i},0)}(\tau^{\bullet}))=-\overline{DT}^{\beta_{i}}(\tau)\tilde{\lambda}^{(\beta_{i},0)}
\end{equation}
where $\overline{DT}^{\beta_{i}}(\tau)$ is the generalized Donaldson--Thomas invariant defined in \cite[Definition 5.15]{a30}. Now apply the Lie algebra morphism $\tilde{\Psi}^{\mathcal{B}_{p}}$ to both sides of Equation \eqref{epsilon2} and use Definition \ref{Bp-ss-defn} as well as the results obtained in \eqref{0,1} and \eqref{1,0}. We obtain the following equation:
\begin{align}\label{wall-cross}
&
\textbf{B}^{ss}_{p}(X,\beta,2,\tilde{\tau})\cdot\tilde{\lambda}^{(\beta,2)}=
\sum_{1\leq l,\beta_{1}+\cdots+\beta_{l}=\beta}\frac{(-1)^{l}}{l!}\cdot[[\cdots[[\tilde{\Psi}^{\mathcal{B}_{p}}(\bar{\epsilon}^{(0,2)}(\tau^{\bullet})),-\overline{DT}^{\beta_{1}}(\tau)\tilde{\lambda}^{(\beta_{1},0)}],-\notag\\
&
\overline{DT}^{\beta_{2}}(\tau)\tilde{\lambda}^{(\beta_{2},0)}],
\cdots],-\overline{DT}^{\beta_{l}}(\tau)\tilde{\lambda}^{(\beta_{l},0)}]\notag\\
\end{align}
\subsubsection{Computation of $\tilde{\Psi}^{\mathcal{B}_{p}}(\bar{\epsilon}^{(0,2)}(\tau^{\bullet}))$}\label{sec11}
By part $(4)$ of Proposition \ref{stackprop} the characteristic stack function of the moduli stack of strictly $\tau^{\bullet}$-semistable objects in class $(0,2)$ is given by:$$\bar{\delta}^{(0,2)}(\tau^{\bullet})=\bar{\delta}(\mathfrak{M}^{(0,2)}_{\mathcal{B}_{p},ss}(\tau^{\bullet}))=\left[\left(\left[\frac{\operatorname{Spec}(\mathbb{C})}{\operatorname{GL}_{2}(\mathbb{C})}\right], \mu\right)\right],$$where $\mu:\left[\frac{\operatorname{Spec}(\mathbb{C})}{\operatorname{GL}_{2}(\mathbb{C})}\right]\to \mathfrak{M}^{(\beta,2)}_{\mathcal{B}_{p},ss}(\tau^{\bullet})$ is the natural embedding morphism.
It is shown in \cite[Section 6.2]{a33} that given a stack function $\left[\left(\left[\frac{\mathcal{U}}{\operatorname{GL}_{2}(\mathbb{C})}\right],\mu\right)\right]$, where $\mathcal{U}$ is a quasi-projective variety, one has the following identity:
\begin{align}\label{FGTG}
\left[\left(\left[\frac{\mathcal{U}}{\operatorname{GL}_{2}(\mathbb{C})}\right],\mu\right)\right]&=
F(\operatorname{GL}_{2}(\mathbb{C}),\mathbb{G}^{2}_{m},\mathbb{G}^{2}_{m})\left[\left(\left[\frac{\mathcal{U}}{\mathbb{G}^{2}_{m}}\right],\mu\circ i_{1}\right)\right]\notag\\
&
+F(\operatorname{GL}_{2}(\mathbb{C}),\mathbb{G}^{2}_{m},\mathbb{G}_{m})\left[\left(\left[\frac{\mathcal{U}}{\mathbb{G}_{m}}\right],\mu\circ i_{2}\right)\right],\notag\\
\end{align}
where and $\mu\circ i_{1}$ and $\mu\circ i_{2}$ are the obvious embeddings and
\begin{align}\label{ration-value}
&
F(\operatorname{GL}_{2}(\mathbb{C}),\mathbb{G}^{2}_{m},\mathbb{G}^{2}_{m})=\frac{1}{2}\,\,\,\,,\,
F(\operatorname{GL}_{2}(\mathbb{C}),\mathbb{G}^{2}_{m},\mathbb{G}_{m})=-\frac{3}{4}.\notag\\
\end{align}
Now substitute  \eqref{ration-value} in \eqref{FGTG} and obtain:
\begin{equation}\label{3/4}
\bar{\delta}^{(0,2)}(\tau^{\bullet})=\frac{1}{2}\left[\left(\left[\frac{\operatorname{Spec}(\mathbb{C})}{\mathbb{G}^{2}_{m}}\right],\mu\circ i_{1}\right)\right]-\frac{3}{4}\left[\left(\left[\frac{\operatorname{Spec}(\mathbb{C})}{\mathbb{G}_{m}}\right],\mu\circ i_{2}\right)\right].
\end{equation}
In order to compute $\tilde{\Psi}^{\mathcal{B}_{p}}(\bar{\epsilon}^{(0,2)}(\tau^{\bullet}))$ one uses the definition of $\bar{\epsilon}^{(0,2)}(\tau^{\bullet})$ in \cite{a30} (Definition 3.10):
\begin{equation}\label{e-total}
\bar{\epsilon}^{(0,2)}(\tau^{\bullet})=\bar{\delta}^{(0,2)}(\tau^{\bullet})-\frac{1}{2}\cdot\bar{\delta}^{(0,1)}(\tau^{\bullet})*\bar{\delta}^{(0,1)}(\tau^{\bullet}).
\end{equation}
Substitute the right hand side of \eqref{3/4} in \eqref{e-total} and obtain:
\begin{align}\label{e-02}
\bar{\epsilon}^{(0,2)}(\tau^{\bullet})&=\frac{1}{2}\left[\left(\left[\frac{\operatorname{Spec}(\mathbb{C})}{\mathbb{G}^{2}_{m}}\right],\mu\circ i_{1}\right)\right]\notag\\
&
-\frac{3}{4}\left[\left(\left[\frac{\operatorname{Spec}(\mathbb{C})}{\mathbb{G}_{m}}\right],\mu\circ i_{2}\right)\right]
-\frac{1}{2}\cdot\bar{\delta}^{(0,1)}(\tau^{\bullet})*\bar{\delta}^{(0,1)}(\tau^{\bullet}).\notag\\
\end{align}
Next we compute $\bar{\delta}^{(0,1)}(\tau^{\bullet})*\bar{\delta}^{(0,1)}(\tau^{\bullet})$ (which is equal to $\bar{\epsilon}^{(0,1)}(\tau^{\bullet})*\bar{\epsilon}^{(0,1)}(\tau^{\bullet})$ since there exist no strictly $\tau^{\bullet}$-semistable objects in $\mathcal{B}_{p}$ with class $(0,1)$).
\subsubsection{Computation of $\bar{\epsilon}^{(0,1)}(\tau^{\bullet})*\bar{\epsilon}^{(0,1)}(\tau^{\bullet})$}\label{compute}
By Proposition \ref{stackprop} $$\bar{\epsilon}^{(0,1)}(\tau^{\bullet})=\left[\operatorname{Spec}(\mathbb{C})/\mathbb{G}_{m}\right].$$In order to compute $\bar{\epsilon}^{(0,1)}(\tau^{\bullet})*\bar{\epsilon}^{(0,1)}(\tau^{\bullet})$, we need to follow the definition of multiplication in the Ringel-Hall algebra of stack functions given in \cite[Definition 2.7]{a30}. Consider objects $(F_{i},V_{i},\phi_{i})$ in $\mathcal{B}_{p}$ of type $(\beta_{i},d_{i})$ for $i=1,\cdots,3$. Let  $\mathfrak{Exact}_{\mathcal{B}_{p}}$ denote the moduli stack of exact sequences of objects in $\mathcal{B}_{p}$ of the form:
\begin{equation}\label{exact1'}
0\rightarrow (F_{1},V_{1},\phi_{1})\rightarrow (F_{2},V_{2},\phi_{2})\rightarrow (F_{3},V_{3},\phi_{3})\rightarrow 0
\end{equation}
Let $\pi_{i}:\mathfrak{Exact}_{\mathcal{B}_{p}}\rightarrow \mathfrak{M}_{B_{p}}(\tau^{\bullet})$ for $i=1,2,3$ be the projection map that sends the exact sequence in \eqref{exact1'} to the left, middle and right terms respectively. Denote by $\mathcal{Z}$ the fibered product $$(\left[\operatorname{Spec}(\mathbb{C})/\mathbb{G}_{m}\right]\times \left[\operatorname{Spec}(\mathbb{C})/\mathbb{G}_{m}\right])\times_{\rho_{1}\times\rho_{1},\mathfrak{M}_{\mathcal{B}_{p}}(\tau^{\bullet})\times \mathfrak{M}_{\mathcal{B}_{p}}(\tau^{\bullet}),\pi_{1}\times\pi_{3}}\mathfrak{Exact}_{\mathcal{B}_{p}}.$$According to \cite[Definition 2.7]{a30}, $\bar{\epsilon}^{(0,1)}(\tau^{\bullet})*\bar{\epsilon}^{(0,1)}(\tau^{\bullet})=(\pi_{2}\circ \Phi)_{*}\mathcal{Z}$ where the map $\Phi$ is given by the following commutative diagram:
\begin{equation*}\label{embedding}
\begin{tikzpicture}
back line/.style={densely dotted}, 
cross line/.style={preaction={draw=white, -, 
line width=6pt}}] 
\matrix (m) [matrix of math nodes, 
row sep=3em, column sep=3.em, 
text height=1.5ex, 
text depth=0.25ex]{ 
\mathcal{Z}&\mathfrak{Exact}_{\mathcal{B}_{p}}&\mathfrak{M}_{\mathcal{B}_{p}}(\tau^{\bullet})\\
\left[\operatorname{Spec}(\mathbb{C})/\mathbb{G}_{m}\right]\times \left[\operatorname{Spec}(\mathbb{C})/\mathbb{G}_{m}\right]&\mathfrak{M}_{\mathcal{B}_{p}}(\tau^{\bullet})\times \mathfrak{M}_{\mathcal{B}_{p}}(\tau^{\bullet})&\\};
\path[->]
(m-1-1) edge node [above] {$\Phi$} (m-1-2)
(m-1-2) edge node [above] {$\pi_{2}$} (m-1-3)
(m-1-1) edge (m-2-1)
(m-1-2) edge node [right] {$\pi_{1}\times \pi_{3}$}(m-2-2)
(m-2-1) edge node [above] {$\rho_{1}\times \rho_{1}$} (m-2-2);
\end{tikzpicture}
\end{equation*} 
\begin{lemma}\label{a1prod}
The product $\bar{\epsilon}^{(0,1)}(\tau^{\bullet}) * \bar{\epsilon}^{(0,1)}(\tau^{\bullet})$ is given as 
\begin{equation*}
\bar{\epsilon}^{(0,1)}(\tau^{\bullet}) * \bar{\epsilon}^{(0,1)}(\tau^{\bullet})=\left[\left(\frac{\operatorname{Spec}(\mathbb{C})}{\mathbb{A}^{1}\rtimes\mathbb{G}^{2}_{m}},\iota\right)\right]
\end{equation*}
where $\iota$ is defined to be the corresponding embedding.
\end{lemma}
\begin{proof}
This is essentially proved in \cite[Lemma 5.3]{a37}.
 \end{proof}

Following \cite[Equation (11.13) , Equation (11.4)]{a30} and Lemma \ref{a1prod}:
\begin{align}\label{e1e1}
\bar{\epsilon}^{(0,1)}(\tau^{\bullet}) * \bar{\epsilon}^{(0,1)}(\tau^{\bullet})&=\left[\left(\frac{\operatorname{Spec}(\mathbb{C})}{\mathbb{A}^{1}\rtimes\mathbb{G}^{2}_{m}},\iota\right)\right]=\notag\\
&
-\left[\left(\frac{\operatorname{Spec}(\mathbb{C})}{\mathbb{G}_{m}},e_{2}\right)\right]+\left[\left(\frac{\operatorname{Spec}(\mathbb{C})}{\mathbb{G}^{2}_{m}},e_{1}\right)\right],\notag\\
\end{align}
where $e_{1}=\mu\circ i_{1}$ and $e_{2}=\mu\circ i_{2}$ denote the corresponding embedding maps. Since $\bar{\epsilon}^{(0,1)}(\tau^{\bullet}) * \bar{\epsilon}^{(0,1)}(\tau^{\bullet})=\bar{\delta}^{(0,1)}(\tau^{\bullet}) * \bar{\delta}^{(0,1)}(\tau^{\bullet})$, by substituting the right hand side of \eqref{e1e1} in \eqref{e-02} one obtains:
\begin{align}\label{e-03}
\bar{\epsilon}^{(0,2)}(\tau^{\bullet})&=\frac{1}{2}\left[\left(\left[\frac{\operatorname{Spec}(\mathbb{C})}{\mathbb{G}^{2}_{m}}\right],\mu\circ i_{1}\right)\right]-\frac{3}{4}\left[\left(\left[\frac{\operatorname{Spec}(\mathbb{C})}{\mathbb{G}_{m}}\right],\mu\circ i_{2}\right)\right]\notag\\
&
-\frac{1}{2}\left(-\left[\left(\frac{\operatorname{Spec}(\mathbb{C})}{\mathbb{G}_{m}},\mu\circ i_{2}\right)\right]+\left[\left(\frac{\operatorname{Spec}(\mathbb{C})}{\mathbb{G}^{2}_{m}},\mu\circ i_{1}\right)\right]\right)\notag\\
&
=-\frac{1}{4}\left[\left(\left[\frac{\operatorname{Spec}(\mathbb{C})}{\mathbb{G}_{m}}\right],\mu\circ i_{2}\right)\right].\notag\\
\end{align}
Now apply the Lie algebra morphism $\tilde{\Psi}^{\mathcal{B}_{p}}$ to $\bar{\epsilon}^{(0,2)}(\tau^{\bullet})$. By Equation \eqref{e-03}:
\begin{align*}
&
\tilde{\Psi}^{\mathcal{B}_{p}}(\bar{\epsilon}^{(0,2)}(\tau^{\bullet}))=\chi^{na}(\frac{-1}{4}\left[\frac{\operatorname{Spec}(\mathbb{C})}{\mathbb{G}_{m}}\right],(\mu\circ i_{2})^{*}\nu_{\mathfrak{M}^{(0,2)}_{\mathcal{B}_{p}}})\tilde{\lambda}^{(0,2)}.
\end{align*} 
Note that by Proposition \ref{stackprop}, $\mathfrak{M}^{(0,2)}_{\mathcal{B}_{p},ss}(\tau^{\bullet})\cong [\operatorname{Spec}(\mathbb{C})/\operatorname{GL}_{2}(\mathbb{C})]$ and hence $\left[\frac{\operatorname{Spec}(\mathbb{C})}{\mathbb{G}_{m}}\right]$ has relative dimension 3 over $\mathfrak{M}^{(0,2)}_{\mathcal{B}_{p},ss}(\tau^{\bullet})$. Moreover, $\left[\frac{\operatorname{Spec}(\mathbb{C})}{\mathbb{G}_{m}}\right]$ is given by  a single point with Behrend's multiplicity $-1$ and $$(\mu\circ i_{2})^{*}\nu_{\mathfrak{M}^{(0,2)}_{\mathcal{B}_{p}}})\tilde{\lambda}^{(0,2)}=(-1)^{3}\cdot\nu_{\left[\frac{\operatorname{Spec}(\mathbb{C})}{\mathbb{G}_{m}}\right]},$$therefore: 
\begin{align*}
&
\tilde{\Psi}^{\mathcal{B}_{p}}(\bar{\epsilon}^{(0,2)}(\tau^{\bullet}))=\chi^{na}\left(\frac{-1}{4}\left[\frac{\operatorname{Spec}(\mathbb{C})}{\mathbb{G}_{m}}\right],(-1)^{3}\cdot\nu_{\left[\frac{\operatorname{Spec}(\mathbb{C})}{\mathbb{G}_{m}}\right]} \right)\tilde{\lambda}^{(0,2)}
\notag\\
&
=(-1)^{1}\cdot(-1)^{3}\cdot \frac{-1}{4}\tilde{\lambda}^{(0,2)}=\frac{-1}{4}\tilde{\lambda}^{(0,2)}.
\end{align*}
The wall-crossing identity \eqref{wall-cross} is then given as follows:
\begin{align}\label{notyet}
&
\textbf{B}^{ss}_{p}(X,\beta,2,\tilde{\tau})\cdot\tilde{\lambda}^{(\beta,2)}
=\notag\\
&
\sum_{1\leq l,\beta_{1}+\cdots+\beta_{l}=\beta}\frac{-1}{4}\cdot\frac{1}{l!}\prod_{i=1}^{l}\overline{DT}^{\beta_{i}}(\tau)\cdot[[\cdots[[\tilde{\lambda}^{(0,2)},\tilde{\lambda}^{(\beta_{1},0)}],\tilde{\lambda}^{(\beta_{2},0)}],\cdots],\tilde{\lambda}^{(\beta_{l},0)}].\notag\\
\end{align}
Now we use the fact that by Definition \ref{13.11}, the generators $\tilde{\lambda}^{(\beta,d)}$ satisfy the following property:
\begin{equation*}
[\tilde{\lambda}^{(\beta,d)},\tilde{\lambda}^{(\gamma,e)}]=(-1)^{\bar{\chi}_{\mathcal{B}_{p}}((\beta,d),(\gamma,e))}\bar{\chi}_{\mathcal{B}_{p}}((\beta,d),(\gamma,e))\tilde{\lambda}^{(\beta+\gamma,d+e)}
\end{equation*}
which enables us to simplify \eqref{notyet} as follows: 
\begin{align}\label{notyet2}
&
\textbf{B}^{ss}_{p}(X,\beta,2,\tilde{\tau})\cdot\tilde{\lambda}^{(\beta,2)}=\sum_{1\leq l,\beta_{1}+\cdots+\beta_{l}=\beta}\frac{-1}{4}\cdot\frac{1}{l!}\cdot\prod_{i=1}^{l}\left(\overline{DT}^{\beta_{i}}(\tau)\cdot \bar{\chi}_{\mathcal{B}_{p}}((\beta_{1}+\cdots+\beta_{i-1},2),(\beta_{i},0))\right)\notag\\
 &
\cdot (-1)^{\sum_{i=1}^{l}\bar{\chi}_{\mathcal{B}_{p}}((\beta_{1}+\cdots \beta_{i-1},2),(\beta_{i},0))}\cdot \tilde{\lambda}^{(\beta,2)}\notag\\
\end{align}
by canceling $\tilde{\lambda}^{(\beta,2)}$ from both sides we obtain the wallcrossing equation and this finishes our computation:
\begin{align}\label{yet}
&
\textbf{B}^{ss}_{p}(X,\beta,2,\tilde{\tau})=
\sum_{1\leq l,\beta_{1}+\cdots+\beta_{l}=\beta}\frac{-1}{4}\cdot\frac{1}{l!}\cdot\prod_{i=1}^{l}\bigg(\overline{DT}^{\beta_{i}}(\tau)\cdot \bar{\chi}_{\mathcal{B}_{p}}((\beta_{1}+\cdots+\beta_{i-1},2)\notag\\
&
,(\beta_{i},0))\cdot (-1)^{\sum_{i=1}^{l}\bar{\chi}_{\mathcal{B}_{p}}((\beta_{1}+\cdots \beta_{i-1},2),(\beta_{i},0))}\bigg).\notag\\
\end{align}
\section{$\tilde{\tau}$-semistable objects in $\mathcal{B}_{p}$ versus higher rank Joyce--Song $\Hat{\tau}$-stable pairs}
Following Remark \ref{better}, the final step in our calculation corresponds to relating the invariants of $\tilde{\tau}$-semistable objects in the auxiliary category $\mathcal{B}_{p}$ to the $\Hat{\tau}$-stable higher rank Joyce--Song pairs. 

\begin{theorem}\label{HFT-Bp}
A higher rank pair $\mathcal{O}^{\oplus r}_{X}(-n)\xrightarrow{\phi} F$ is $\Hat{\tau}$-(semi)stable in the sense of Definition \ref{newstabl} if and only if it is $\tilde{\tau}$-semistable (when realized as an object in $\mathcal{B}^{\textbf{R}}_{p}$) 
\end{theorem}
\begin{proof}
1. $\Hat{\tau}$-semistability$\Rightarrow \tilde{\tau}$-semistability:

One proves the claim by contradiction. Suppose $\mathcal{O}^{\oplus r}_{X}(-n)\rightarrow F$ is not $\tilde{\tau}$-semistable. Then there exists a subobject $\mathcal{O}^{\oplus r}_{X}(-n)\rightarrow F'$, a quotient object $(Q,0,0)$ and an exact sequence$$0\rightarrow [\mathcal{O}^{\oplus r}_{X}(-n)\rightarrow F']\rightarrow [\mathcal{O}^{\oplus r}_{X}(-n)\rightarrow F]\rightarrow [0\rightarrow Q]\rightarrow 0,$$such that $\tilde{\tau}(\mathcal{O}^{\oplus r}_{X}(-n)\rightarrow F')=1$ and $\tilde{\tau}(0\rightarrow Q)=0$. Consider the following commutative diagram:
\begin{equation}\label{tilde-diag}
\begin{tikzpicture}
back line/.style={densely dotted}, 
cross line/.style={preaction={draw=white, -, 
line width=6pt}}] 
\matrix (m) [matrix of math nodes, 
row sep=1em, column sep=5.5em, 
text height=1.5ex, 
text depth=0.25ex]{ 
0&0\\
\mathcal{O}^{\oplus r}_{X}(-n)&F'\\
\mathcal{O}^{\oplus r}_{X}(-n)&F\\
0&Q\\
0&0\\};
\path[->]
(m-1-1) edge (m-2-1) 
(m-2-1) edge (m-2-2) edge node [left] {$id$} (m-3-1)
(m-3-1) edge (m-3-2) edge (m-4-1)
(m-4-1) edge (m-4-2)
(m-1-2) edge (m-2-2);
\path[right hook->] (m-2-2) edge (m-3-2);
\path[->] (m-3-2) edge (m-4-2)
(m-4-1) edge (m-5-1)
(m-4-2) edge (m-5-2);
\end{tikzpicture}.
\end{equation} 
Now consider the right vertical short exact sequence in diagram \eqref{tilde-diag}. Since $\mathcal{A}_{p}$ is an abelian category (it contains kernels and cokernels), we can see that naturally $p(F')=p\nless p(F)=p$, and we obtain a contradiction with $\Hat{\tau}$-stability of $\mathcal{O}^{\oplus r}_{X}(-n)\rightarrow F$ according to part (2) of Definition \ref{newstabl}.

(2). $\tilde{\tau}$-semistability$\Rightarrow \Hat{\tau}$-stability

Similarly suppose $\mathcal{O}^{\oplus r}_{X}(-n)\xrightarrow{\phi} F$ is $\Hat{\tau}$-unstable. Then there exists a proper nonzero subsheaf $F'\subset F$ such that $\phi$ factors through $F'$ and $p(F')=p(F)=p$. Now obtain the diagram in \eqref{tilde-diag} and consider the right vertical short exact sequence. By the same reasoning as above, $p(Q)=p$. Hence $Q\in\mathcal{A}_{p}$ and the complex $0\rightarrow Q$ represents an object in $\mathcal{B}_{p}$ with $\tilde{\tau}(0\to Q)=0<\tilde{\tau}(\mathcal{O}^{\oplus r}_{X}(-n)\xrightarrow{\phi} F')=1$. Hence $\mathcal{O}^{\oplus r}_{X}(-n)\rightarrow F$ is not $\tilde{\tau}$-semistable according to Definition \ref{weakjoyce}, which contradicts the assumption.
\end{proof}
\begin{cor}\label{N}
The moduli stack, $\mathfrak{M}^{(\beta,r)}_{stp}(\Hat{\tau})$,  of $\Hat{\tau}$-semistable pairs of rank $r$ is a $\operatorname{GL}_{r}(\mathbb{C})$-torsor over $\mathfrak{M}^{(\beta,r)}_{\mathcal{B}_{p},ss}(\tilde{\tau})$ with a local section map $\pi: \mathfrak{M}^{(\beta,r)}_{\mathcal{B}_{p},ss}(\tilde{\tau})\to\mathfrak{M}^{(\beta,r)}_{stp}(\Hat{\tau})$ of relative dimension $-r^2$. In particular it is true that $\pi^*\nu_{\mathfrak{M}^{(\beta,r)}_{stp}(\Hat{\tau})}=(-1)^{r^2}\cdot\nu_{\mathfrak{M}^{(\beta,r)}_{\mathcal{B}_{p},ss}(\tilde{\tau})}$.

\end{cor} \begin{proof}
The proof of the statement in the corollary is given as a direct consequence of Theorem \ref{HFT-Bp} and Theorem \ref{theorem81}; The consequence of Theorem \ref{HFT-Bp} is that, there exists an isomorphism between the moduli stack $\mathfrak{M}^{(\beta,r)}_{stp}(\Hat{\tau})$ and the part of $\mathfrak{M}^{(\beta,r)}_{\mathcal{B}^{\textbf{R}}_{p},ss}$ that lives over $\mathfrak{M}^{(\beta,r)}_{\mathcal{B}_{p},ss}(\tilde{\tau})$. Therefore, using this isomorphism, together with the same arguments as in proof of parts (2) and (3) of Theorem  \ref{theorem81} we see that for $r>1$, $\mathfrak{M}^{(\beta,r)}_{stp}(\Hat{\tau})$ makes a $\operatorname{GL}_{r}(\mathbb{C})$-torsor over $\mathfrak{M}^{(\beta,r)}_{\mathcal{B}_{p},ss}(\tilde{\tau})$ and moreover, locally $\mathfrak{M}^{(\beta,r)}_{\mathcal{B}_{p},ss}(\tilde{\tau})$ is a $\operatorname{GL}_{r}(\mathbb{C})$-gerbe over $\mathfrak{M}^{(\beta,r)}_{stp}(\Hat{\tau})$, and hence the existence of an identity between the corresponding Behrend functions follows immediately from this fact. This is similar to the statement in \cite[Proposition 13.6 (b)]{a30} where the authors showed that for $r=1$, $\mathfrak{M}^{(\beta,1)}_{\mathcal{B}_{p},ss}(\tilde{\tau})$ is a $\mathbb{G}_{m}$-gerbe over $\mathfrak{M}^{(\beta,1)}_{stp}(\Hat{\tau})$, and they concluded a similar identity between the Behrend functions.
\end{proof}
 Let us now recall the definition of the invariants of rank $r$ $\Hat{\tau}$-semistable pairs one more time$$\textbf{N}^{\beta,r}_{stp}(\hat{\tau})=(-1)^{r^2}\textbf{B}_{p}^{ss}(X,\beta, r, \tilde{\tau}).$$
 Now we specialize this definition to $r=2$ and obtain; 
\begin{cor}\label{he-lol}
Set $r=2$. In this case $\nu_{\mathfrak{M}^{(\beta,2)}_{stp}(\Hat{\tau})}=\nu_{\mathfrak{M}^{(\beta,r)}_{\mathcal{B}_{p},ss}(\tilde{\tau})}$ by Corollary \ref{N}. Now we use this fact and define the invariant of rank 2 $\Hat{\tau}$-semistable pairs with $\text{ch}(F)=\beta$ to be equal to $\textbf{B}^{ss}_{p}(X,\beta,2,\tilde{\tau})$ defined in Definition \ref{Bp-ss-defn}. In other words we use the following identity as the definition of $\textbf{N}^{\beta,2}_{stp}(\Hat{\tau})$: 
\begin{align}\label{define-it}
&
\textbf{N}^{\beta,2}_{stp}(\Hat{\tau})=
\sum_{1\leq l,\beta_{1}+\cdots+\beta_{l}=\beta}\frac{-1}{4}\cdot\frac{1}{l!}\cdot\prod_{i=1}^{l}\bigg(\overline{DT}^{\beta_{i}}(\tau)\cdot \bar{\chi}_{\mathcal{B}_{p}}((\beta_{1}+\cdots+\beta_{i-1},2)\notag\\
&
,(\beta_{i},0))\cdot (-1)^{\sum_{i=1}^{l}\bar{\chi}_{\mathcal{B}_{p}}((\beta_{1}+\cdots \beta_{i-1},2),(\beta_{i},0))}\bigg).\notag\\
\end{align}
\end{cor}
\bibliographystyle{plain}
\bibliography{ref1}
\end{document}